\documentclass[11pt]{article}
\usepackage{amsmath}
\usepackage{amssymb}
\usepackage{latexsym}
\usepackage{theorem}
\usepackage{mathrsfs}
\usepackage[all]{xy}
\usepackage{orthog}
\usepackage{stmaryrd}
\usepackage[bookmarksnumbered]{hyperref}

\begin{document}

\title{Model Categories for Orthogonal Calculus}

\author{David Barnes         \and
        Peter Oman 
}


%

\maketitle

\begin{abstract}
\noindent We restate the notion of orthogonal calculus in 
terms of model categories.  This provides a cleaner set of results 
and makes the role of $O(n)$--equivariance clearer. 
Thus we develop model structures for the category of $n$--polynomial
and $n$--homogeneous functors, along with Quillen pairs relating them. 
We then classify  $n$--homogeneous functors, via a zig-zag of 
Quillen equivalences,  in terms of 
spectra with an $O(n)$--action. 
This improves upon the classification theorem of Weiss.
As an application,  we develop a variant of orthogonal calculus by replacing topological spaces with orthogonal spectra.
\end{abstract}

\section{Introduction}

Orthogonal calculus is a beautiful tool for calculating homotopical properties
of a functor from the category of finite dimensional inner product spaces and linear 
isometries to the category of based spaces.  Interesting examples of such functors abound and include classical objects in algebraic and geometric topology: the classifying space of the orthogonal group, $F(V)=BO(V)$;  the classifying space of the group of homeomorphisms of $V$, $F(V)=B\text{Top}(V)$; and the space of  Euclidean embeddings for a fixed manifold $M$, $F(V)=\text{Emb}(M,V)$.  
We call the category of such functors and natural transformations between them 
$\ecal_0$. Orthogonal calculus is based on the notion of 
$n$--polynomial functors, which are well-behaved functors in $\ecal_0$. These functors 
satisfy an extrapolation condition, which allows one to identify the value at some 
vector space from the values at vector spaces of greater dimension, see section \ref{sec:npolyfun}. 
In geometric terms, orthogonal calculus approximates a functor (locally around $\rr^{\infty}$)  via polynomial functors and attempts to reconstruct the global functor from the associated ``infinitesimal" information.
More concretely, the calculus splits a functor $F$ in $\ecal_0$ into a 
tower of fibrations, the $n^{th}$--fibration of this tower consists of a map from the 
$n$--polynomial approximation of $F$ to the $(n-1)$--polynomial
approximation of $F$. The homotopy fibre of this map is then an $n$--homogeneous functor and is classified by an $O(n)$--spectrum up to homotopy. 

A question about an input functor can thus be translated into a question about the spectra and the fibration sequences they are part of, hopefully making the answer easier to find. Indeed, from a computational perspective, the orthogonal tower provides a spectral sequence 
the inputs of which have more structure and more constraints than the original functor.  
A particularly interesting example of how useful orthogonal calculus can be is \cite{ARV07}. 
In that paper, orthogonal calculus is used to prove that the rational homology type of certain spaces of embeddings of a manifold 
is determined by the rational homology type of the manifold.  In addition, the homogeneous approximations themselves 
typically correspond to spectra of intrinsic interest. For example, in \cite{weiss95}, Weiss determines that the 1-homogeneous approximation of the functor $BO(-)$ is the sphere spectrum and the 1-homogeneous approximation of $B\text{Top}(-)$ is Waldhausen's A-theory of a point.

In the construction of the orthogonal calculus, \cite{weiss95}, there are numerous times that one wishes to study a functor `up to homotopy' for some (usually non-standard) notion of homotopy. Furthermore, the classification of the $n^{th}$--fibre of the tower is given in terms of an equivalence of homotopy categories. 
We replace `up to homotopy' by some appropriate notion of weak equivalence within a model structure. This allows us to 
replace an equivalence of homotopy categories by the stronger and more structured notion of a Quillen equivalence: exhibiting an equivalence of `homotopy theories' rather than merely an equivalence of derived categories. 

In \cite{weiss95}, Weiss implicitly constructs a localisation of the standard projective model structure, capturing the homotopy theory of $n$--polynomial functors.  One would expect that such a localization can be realized in model theoretic language and indeed this is the case.
Via a homotopy-idempotent functor, we can left Bousfield localize to create a (right proper) model structure whose fibrant objects are 
the $n$--polynomial functors.  We call this model category $n \poly \ecal_0$.
However, subtleties arise when attempting to construct a model structure capturing the homotopy theory of $n$--homogeneous functors and the associated classification by $O(n)$-spectra.

A priori, the $n$--homogeneous structure is a right localisation of the $n$--polynomial structure. Since we not working stably, this requires the full Bousfield localisation machinery in the sense of Hirschhorn \cite{hir03}.
Indeed, in order to know that the $n$--polynomial model structure is well-behaved enough to admit a right localisation (specifically, that it is a cellular model category), we must make use of this general machinery to provide a second construction of  $n \poly \ecal_0$.

These two constructions show that $n \poly \ecal_0$ is right proper cellular model category. 
Thus we are able to perform a 
right Bousfield localisation to obtain a new model category $n \homog \ecal_0$
whose cofibrant--fibrant objects are precisely the $n$--homogeneous functors. 
The task is now to identify this complicated model structure as something simpler. 
With the work of Weiss, one would hope to prove a Quillen equivalence
between this model category and the category of $O(n)$--objects in spectra; however, we show $O(n)$--objects in spectra are not the most natural models for homogeneous functors. 

By carefully examining the derived equivalence in \cite{weiss95}, we construct
a non--standard model of the (free) $O(n)$--equivariant stable category which we call $O(n) \ecal_n$. 
This category captures the appropriate structure for differentiation, see section \ref{sec:thecat}. 
Indeed, we exhibit a left Quillen equivalence from $O(n) \ecal_n$ to the category of orthogonal spectra with an $O(n)$--action. 
We then show that homogeneous functors are classified, via differentiation as a right Quillen functor, by objects of $O(n) \ecal_n$.  Thus we recover the derived equivalence of Weiss via a zig-zag of Quillen equivalences and gain a much better understanding of the role of equivariance within the theory. 
In addition,  it may be of technical interest that this is an example of a twofold (left and right) localisation of a cellular model category where one has a nice description of the resulting homotopy theory; in general, one would expect such a thing to be quite unwieldy.

In more detail, the category $O(n) \ecal_n$ is best thought of as a variant of the theory of $O(n)$--equivariant orthogonal spectra, where a universe with trivial $O(n)$-action is used and the sphere spectrum is replaced by the $O(n)$--equivariant object which at $V$ takes value $S^{nV}$, the one-point compactification of $\rr^n \otimes V$. From such an object one can obtain an object of $\ecal_0$ by neglect of structure and taking orbits, we call this composite functor $\res_0^n/O(n)$ and we examine it in more detail in sections 
\ref{sec:diff} and \ref{sec:infind}. When $n >0$, we put a stable model structure on $O(n) \ecal_n$ that is a variation of the usual model structure on spectra, adjusted to account for our non-standard `sphere spectrum', see 
section \ref{sec:nstab}.

The functor $V \mapsto S^{nV}$ was already known to occur in orthogonal calculus.
From a spectrum with $O(n)$--action $X$, one can construct 
an $n$--homogeneous functor by the following rule
\[
V \longmapsto EO(n)_+ \smashprod_{O(n)} [\Omega^\infty (X \smashprod S^{nV})]
\]
where the term $X \smashprod S^{nV}$ is equipped with the diagonal action. 
We are able to replace this construction with the pair of Quillen equivalences as below, 
see sections \ref{sec:equiv} and \ref{sec:tower}.
The first category is the model category of $n$--homogeneous functors from 
vector spaces and linear isometries to based spaces, 
the second our new category and the third the model category 
of $O(n)$--objects in spectra. 
$$\xymatrix@C+1cm{
n \homog \ecal_0
\ar@<-1 ex>[r]_(0.6){\ind_0^n \varepsilon^*}
&
O(n)\ecal_n^{\pi}
\ar@<-1ex>[l]_(0.4){\res_0^n/O(n)}
\ar@<+1ex>[r]^{(-)\smashprod_{\jcal_n}\jcal_1}
&
O(n)\ical \scal
\ar@<+1ex>[l]^{\alpha_n^*}
}
$$
The functor $\alpha_n^*$ takes a spectrum $X$ to the object of $O(n) \ecal_n$
which at $V$ takes value $X(nV)$ with $O(n)$--action given by first applying the 
$O(n)$--action of $X$ and then the $O(n)$--action induced by that on $nV = \rr^n \otimes V$. 

These Quillen equivalences give a much cleaner statement than the homotopy classification of \cite[Section 7]{weiss95}, which was greatly hampered by 
the need to work exclusively with $\Omega$--spectra. They also make the role of the functor $V \mapsto S^{nV}$ clearer. Indeed, much of the difficulty of this work was to establish the correct categories and functors for the above.

The sets of maps at which we left localise and the set of objects at which we right localise are intrinsically bound into the definition of $O(n) \ecal_n$ and the notion of differentiation (which is embodied by the functor $\ind_0^n \varepsilon^*$). We find these intricate relations quite gratifying to see and perhaps illuminate what kind of definitions one would need to create a new calculus. 

Equally, it would be interesting to see how much of this work can be replicated for the other notions of calculus that use homotopy theory. 
For example, the Goodwillie calculus of functors \cite{gw90,gw91,gw03} was one of the inspirations for orthogonal calculus. It has long been known that they are strongly related, despite their very different inputs, thus it should not be surprising that 
\cite{bcr} follows a similar pattern to our work. 

We conclude this paper with an application, because of the way we have used model categories to develop the theory, we immediately obtain a stable variant of orthogonal calculus by replacing topological spaces with orthogonal spectra, see section \ref{sec:stablecase}. 

Finally, we believe that other extensions and alterations of orthogonal calculus will be much easier to create now that we have a good model category foundation. For example, one could study functors into some localisation of spaces or spectra. Alternatively, one 
could repeat this work for the unitary calculus, where $O(n)$ is replaced by $U(n)$. Indeed, this project began as an attempt to perform an equivariant version of orthogonal calculus. It is now clear that any such attempt will need to have clear and precise categorical constructions, along with well-behaved model categories.

\subsubsection*{Organisation}
We begin in section \ref{sec:groupstuff} with a brief introduction to spaces with a group action
and enriched functors. This provides the language we need 
section \ref{sec:thecat} to define the categories $O(n) \ecal_n$. 
We show how differentiation relates these categories in section \ref{sec:diff}.

The notions of $n$--polynomial and $n$--homogeneous functors are introduced in section \ref
{sec:npolyfun}. The next task, completed in section \ref{sec:npolymodel}, 
is to find a model category in which the fibrant objects are $n$--polynomial functors 
and another where the cofibrant--fibrant objects and the $n$--homogeneous functors. 

Staying with model structures, in section \ref{sec:nstab} we produce an $n$--stable model 
structure on $O(n) \ecal_n$. This will be an intermediary model structure sitting between 
the $n$--homogeneous model structure on $\ecal_0$ and 
the category of $O(n)$--equivariant objects in the category of orthogonal spectra. 
We prove that $O(n) \ecal_n$ with the $n$--stable model structure is Quillen equivalent to 
the model category of $O(n)$--objects in orthogonal spectra in section \ref{sec:equiv}.

In section \ref{sec:infind} we prove that 
the differentiation functor is a right Quillen functor from the $n$--polynomial model category to $O(n) \ecal_n$ equipped with the stable model structure.
We use this to show that $O(n) \ecal_n$ is in fact Quillen equivalent to the $n$--homogeneous model category in section \ref{sec:tower}. As a consequence, we recover the statement that the tower has the desired form and that the homotopy type of the $n^{th}$--fibre is determined by 
a spectrum with an $O(n)$--action. 
We conclude the paper by developing a stable variant of orthogonal calculus in section \ref{sec:stablecase}.

We have tried to make this paper largely self--contained and hence we have reproduced a fair amount of Weiss's work, 
often with improvements in the proofs or descriptions due to our use of model categories. 
There are some areas that we have not been able to improve upon, most notably the homotopy colimit calculation of 
\cite[Theorem 4.1]{weiss95}, the properties of the functor $T_n$, from Theorem 6.3 and the errata, 
and the calculations of Examples 5.7 and 6.4. We find this to be acceptable, 
as the aim of this paper is not to replace \cite{weiss95} but to 
put it into the modern language of model categories.

\section{Group actions and enriched functors}\label{sec:groupstuff}
Since equivariance is vital to our approach, we briefly introduce 
discuss spaces with a group action and functors enriched over spaces with a group action.
The following is a summary of \cite[Section III.1]{mm02}.

In this section $G$ will be a compact Lie group. A \textbf{based $G$--space} $X$
is a topological space with a continuous action of $G$  
on the space $X$. We require that this action be associative 
and unital and that the basepoint of $X$ be fixed by the action of $G$.
A continuous map $f \co X \to Y$ between two $G$--spaces is 
said to be an \textbf{equivariant map}
if $f(gx) = gf(x)$ for any $g \in G$ and $x \in X$. 
We write $G \Top$ for the category of based $G$--spaces 
and equivariant maps

The category $G \Top$ is a closed symmetric monoidal
category, whose \textbf{monoidal product} is given by the smash product of based spaces 
equipped with the diagonal action of $G$. We note that
whenever a smash product of two $G$--spaces appears
in \cite{weiss95}, it is also given the diagonal action.
The corresponding \textbf{internal function object} is
the space of non--equivariant maps, which has a $G$--action 
defined by conjugation. For $G$--spaces $X$ and $Y$, 
we denote this space by $\Top(X,Y)$
and we see that $g \ast f = gfg^{-1}$ for $f \in \Top(X,Y)$
and $g \in G$. 
In  particular, $\Top(X,Y)^G$ is precisely the space
of equivariant maps from $X$ to $Y$. 

This category has a cofibrantly generated proper model structure
where the fibrations and weak equivalences are those equivariant maps 
$f \co X \to Y$ 
whose underlying map of spaces 
$i^* f \co i^*X \to i^*Y $ 
is a fibration or weak homotopy equivalence of spaces. 
The generating cofibrations of this model structure are
given by the standard boundary inclusions:
$$(G \times S^{n-1})_+ \to  (G \times D^{n})_+$$
for $n \geqslant 0$, with the sphere and disc
both given the trivial $G$--action. 
The generating acyclic cofibrations are given by the 
maps 
$$(G \times D^{n})_+ \to  (G \times D^{n} \times [0,1])_+$$
where $(g,x) \mapsto (g,x,0)$ and $n \geqslant 0$. 
Using \cite[Lemma IV.6.6]{mm02} it is easy to check that 
these definitions give
us a symmetric monoidal model structure on the category of $G$--spaces.
Hence the smash product and internal function object have derived functors
on the homotopy category. 

We will also need the language of functors enriched over $G$-spaces, 
so we give an introduction here.  
We take our definitions from 
\cite[Section II.1]{mm02}.

Following the usual convention we call
a space--enriched functor from a topological category $\dcal$ 
to spaces a \textbf{$\dcal$--space}.
A \textbf{map of $\dcal$--spaces} $f \co E \to F$ is then a 
collection of continuous maps $f(d) \co E(d) \to F(d)$
such that for any element $\alpha \in \dcal(d,e)$
we have a commutative square
\[
\xymatrix@C+1cm{
E(d) \ar[r]^{E(\alpha)} \ar[d]^{f(d)} &
E(e) \ar[d]^{f(e)} \\
F(d) \ar[r]^{F(\alpha)} &
F(e)
}
\]

If the category $\dcal$ is enriched over $G$--spaces, then we can also consider 
continuous functors $E$ from $\dcal$ to $G$--spaces such 
that the map
\[
E_{d,e} \co \dcal(d,e) \longrightarrow \Top ( E(d) , E(e) )
\]
is $G$--equivariant. 
We call such a functor a 
\textbf{$G$--equivariant $\dcal$--space}. 
Such functors are precisely the $G \Top$--enriched functors from $\dcal$ to $G \Top$. 
We then define a 
\textbf{map of $G$--equivariant $\dcal$--spaces}, $f \co E \to F$, to be a 
collection of equivariant maps $f(d) \co E(d) \to F(d)$
such that for any element $\alpha \in \dcal(d,e)$
we have a commutative square as before. It is important to 
note that we ask for this diagram to commute 
for any $\alpha$, even though
$E(\alpha)$ or $F(\alpha)$ are not necessarily equivariant maps. 
The category of $G$--equivariant $\dcal$--spaces and maps of
$G$--equivariant $\dcal$--spaces will be denoted $G \dcal \Top$.

Our final piece of business in this section is to 
note that the category of $G$--equivariant $\dcal$--spaces
is itself enriched over $G$--spaces. 
We will use this in section \ref{sec:diff} to define differentiation. 
To describe this enrichment we need to use the notion of enriched ends, 
details can be found in \cite{bor94}. 
\begin{definition}
For $E$ and $F$ in $G \dcal \Top$, define the 
\textbf{$G$--equivariant space of maps} from $E$ to $F$ to be the following 
enriched end (which is constructed in the category $G \Top$). 
$$
\nat_{G \dcal \Top} (E,F) = \int_{d \in \dcal} \Top(E(d), F(d))
$$
\end{definition}
It is routine to show that $\nat_{G \dcal \Top} (E,F)^{G}$
is the space of maps of $G$--equivariant $\dcal$--spaces.

For model category purposes and constructions 
it is helpful to also have a tensor
and cotensor. 
For $E \in G \dcal \Top$
and $A$ a topological space with $G$--action, 
there is another object 
$A \otimes E \in G \dcal \Top$, which at $U$ takes
value $A \smashprod E(U)$. The structure maps
of this new object are given by 
\[
\dcal(d,e) \overset{E_{d,e}}{\longrightarrow}
 \Top(E(d), E(e))  \overset{A \smashprod -}{\longrightarrow}
 \Top(A \smashprod E(d),A \smashprod  E(e)).
\]
There is also a cotensor product, 
$\hom(A,E)$ which at $U$ takes value 
$\Top(A,E(U))$. The structure map for this object is
given below. 
\[
\dcal(d,e) \overset{E_{d,e}}{\longrightarrow}
 \Top(E(d), E(e))  \overset{\Top(A, -)}{\longrightarrow}
 \Top(\Top(A,E(d)),\Top(A,E(e))).
\]

Let $G\dcal \Top(E,F)$ denote the set of 
maps in the category $G \dcal \Top$ between 
two objects of that category. 
Similarly, let $G \Top(A,B)$ denote the set of 
maps in the category $G \Top$ between two objects $A$ and $B$.
Then we can relate the enrichment, tensor and cotensor
by the natural isomorphisms \[
G\dcal \Top(E, \hom(A,F) ) 
\cong 
G\dcal \Top(A \otimes E, F) 
\cong 
G \Top(A, \nat_{G \dcal \Top} (E,F) )
\] 
Thus we see that $G \dcal \Top$ is a closed
module over $G \Top$, in the sense 
of \cite[Section 4.1]{hov99}. 
One can repeat these constructions with the category of 
$\dcal$--spaces and maps of $\dcal$--spaces to see that it is 
a closed module over $\Top$. 

For a fixed $G$--topological category $\dcal$ 
we would like an adjunction comparing
$G \dcal \Top$ and $G \Top$. 
It is clear that any $G$--equivariant $\dcal$--space gives a 
$\Top$--enriched functor from $\dcal$ to based 
topological spaces by forgetting the $G$--actions. 
Hence there is a forgetful
functor from $G \dcal \Top$ to $\dcal \Top$.
We call this functor $i^*$ just as 
the forgetful functor from $G \Top$ to $\Top$
is called $i^*$. 

There is a left adjoint to $i^*$, 
which we write as $G_+ \smashprod -$. 
Let $E \in \dcal \Top$, then at $d$
we define 
\[
(G_+ \smashprod E)(d) = G_+ \smashprod E(d)
\]
The structure map is then defined as follows
\[
(G_+ \smashprod E(d) ) \smashprod \dcal(d,e)
\to
G_+ \smashprod (E(d)  \smashprod i^* \dcal(d,e))
\to 
G_+ \smashprod E(e)
\]
where the first map is an isomorphism given by 
$(g, x,y) \mapsto (g, x, g^{-1} y)$
for $g \in G$, $x \in E(d)$ and 
$y \in \dcal(d,e)$. The second map is then the action map
of $E$. It is routine to check that we have an adjoint pair as claimed.

\section{The category \texorpdfstring{$O(n)\ecal_n$}{O(n)E\_n}}\label{sec:thecat}

Our primary objects of study are continuous functors from
the category of finite dimensional real inner product spaces 
and linear isometries to based spaces. We call this category $\ecal_0$. 
These functors are the input to orthogonal calculus, the output is a 
tower of fibrations, whose fibres are continuous functors like the input 
but have more structure. The main theorem of orthogonal calculus is that these
fibres can be classified in terms of spectra with an $O(n)$--action. 
On the way to this classification we need an intermediate category
$O(n) \ecal_n$. 
In this section we introduce the categories $\ecal_0$ and 
$O(n) \ecal_n$. Since they are both defined in terms
of enriched functors, we will need to construct
a collection of enriched categories $\jcal_n$ for $n \geqslant 0$. 

\begin{definition}
Let $\ucal$ be a real inner product space with countably infinite dimension. 
For $V$ and $W$ finite dimensional inner product subspaces of $\ucal$,
we define $\mor(V,W)$ to be
the Stiefel manifold of linear isometries from $V$ to $W$.

We then define $\jcal_0$ to be the topological category with objects the class of finite dimensional real inner product subspaces of $\ucal$.
The morphism space of maps from $V$ to $W$, written $\mor_0(V,W)$, 
is defined to be $\mor(V,W)_+$
(the space $\mor(V,W)$ with a disjoint basepoint added). 

We define $\ical$ to be the topological category with same objects as $\jcal_0$ but with morphisms given by the space of
linear isometric isomorphisms from $V$ to $W$ (with a disjoint basepoint added).
We write $\ical(V,W)_+$ for this space. 
\end{definition}

Our input functors can then be described as $\jcal_0$--spaces. 
Such an object $F$ consists of a collection of spaces $F(V)$, one for 
each finite dimensional real inner product subspace, along with 
continuous maps 
\[
F(V) \smashprod \mor_0(V,W) \to F(W)
\]
that satisfy an evident associativity condition 
and a unit condition when $V=W$. 

Recall that the fibres of the tower will be classified 
using the category $O(n) \ecal_n$. To construct this category we must
define an enriched category $\jcal_n$, which we build out of some 
$O(n)$--equivariant vector bundles. The category $\jcal_n$ is constructed in 
\cite[Section 1]{weiss95} and is given an $O(n)$--action in 
\cite[Section 3]{weiss95}.

For each pair $V$ and $W$ we define an $O(n)$--equivariant vector 
bundle $\gamma_n(V,W)$ on the space 
$\mor(V,W)$. The total space of this vector bundle is given by 
\[
\left\{ (f,x) \ | \ f \in \mor(V,W) \ \text{ and }
x \in \rr^n \otimes (W - f(V))
\right\}
\]
where $W-f(V)$ denotes the orthogonal
complement of the image of $f$.  
The group $O(n)$ acts on $\rr^n$ 
via linear isometries and we extend this action to 
$\rr^n \otimes (W - f(V))$ by letting $\sigma \in O(n)$ act by 
$\sigma \otimes \id$. Now we let $O(n)$ act on 
$\gamma_n(V,W)$ by the rule $\sigma(f,x) = (f, \sigma \otimes \id (x))$.

\begin{definition}
For each $n \geqslant 0$, the \textbf{$n^{th}$--jet category} $\jcal_n$ is an $O(n)$--topological category with the same class of objects as $\jcal_0$. The space of morphisms from $V$ to $W$ is the Thom space of 
$\gamma_n(V,W)$ and is denoted $\mor_n(V,W)$. 
Note that the point at infinity is fixed under the $O(n)$-action.
We also see that when $n=0$ this definition agrees with our earlier description
of $\jcal_0$.

The composition rule for $\mor_n$ is induced by the map of spaces
\[
\gamma_n(V,W) \times \gamma_n(U,V) \longrightarrow \gamma_n(U,W)
\]
given by the formula $((f,x),(g,y) \mapsto (fg, x + f_*(y))$, where $f \co V \to W$, $g \co U \to V$, $f_*=\id \otimes f$, $x \in \rr^n \otimes (W-f(V))$ and $y \in \rr^n \otimes (V-g(U))$.
\end{definition}

\begin{definition}
We define $\ecal_n$ to be $\jcal_n \Top$, the category of
$\jcal_n$--spaces and maps of $\jcal_n$-spaces. 

We define $O(n)\ecal_n$ to be $O(n) \jcal_n \Top$,
the category of
$O(n)$--equivariant $\jcal_n$--spaces
and maps of 
$O(n)$--equivariant $\jcal_n$--spaces.
\end{definition}

When $n=0$, $O(n)\ecal_n$ is the same as $\ecal_0$, our category of input functors. 

We will later use the categories $O(n) \ecal_n$ to 
classify the fibres of the orthogonal tower. 
For this result it is essential that we use the
$O(n)$--topological enrichment. 
In \cite{weiss95}, only $\ecal_n$ is used, 
but using the adjunction of the previous section
we see that any object of $O(n) \ecal_n$
defines an object of $\ecal_n$ by forgetting the 
$O(n)$-actions.

\section{Differentiation}\label{sec:diff}

Differentiation is a method of taking a functor in $\ecal_0$ 
and making a functor in $O(n) \ecal_n$. 
This process is central to orthogonal calculus
because the $n^{th}$--fibre of the tower for a functor $E \in \ecal_0$
will be determined by the $n^{th}$--derivative of $E$ (up to homotopy). 
We will obtain an adjoint pair 
between $\ecal_0$ and $O(n) \ecal_n$. 
The (derived) counit of this adjunction will then describe
how to include the $n^{th}$--fibre of the tower into the 
$n^{th}$--term. 

For the sake of completeness, we consider a more general version of this adjunction. We define differentiation as a functor
$O(m)\ecal_m \to O(n)\ecal_n$ for $m \leqslant n$. 

Let $i_m^n \co \rr^m \to \rr^n$ be the map $x \mapsto (x,0)$, 
where $m \leqslant n$.
This map induces a group homomorphism $O(m) \to O(n)$, where $O(m)$ acts on the first $m$ coordinates and leaves the rest unchanged. This makes $i_m^n \co \rr^m \to \rr^n$ a map of $O(m)$--equivariant objects. We can also use $i_m^n$ to induce a functor of $O(m)$--topological categories $\jcal_m \to \jcal_n$. To do so, we apply the Thom space construction to the map of $O(m)$--equivariant spaces:
\[
\begin{array}{rcl}
(i_m^n)_{U,V} \co \gamma_m(U,V) &\longrightarrow & \gamma_n(U,V) \\
(f,x) & \longmapsto & (f, i_m^n \otimes \id (x))
\end{array}
\] 
We thus have a series of maps of enriched categories. 
\[
\jcal_0 \overset{i_0^1}{\longrightarrow} 
\jcal_1 \overset{i_1^2}{\longrightarrow} 
\jcal_2 \overset{i_2^3}{\longrightarrow} 
\dots     \overset{i_{n-1}^n}{\longrightarrow} 
\jcal_n \overset{i_n^{n+1}}{\longrightarrow} \dots 
\]
We now study how these maps induce functors between the categories
$\ecal_n$ for varying $n$. By adding in change of groups functors, 
we also achieve adjoint pairs between the categories 
$O(n) \ecal_n$ for varying $n$.

\begin{definition}
Define the \textbf{restriction functor} $res_m^n \co \ecal_n \to \ecal_m$ as precomposition with $i_m^n \co \jcal_m \to \jcal_n$, where $m \leqslant n$.

Similarly define the \textbf{restriction--orbit functor} $res_m^n/O(n-m) \co O(n)\ecal_n \to O(m)\ecal_m$ as the functor which sends $X$ to $(X \circ i_m^n )/O(n-m)$ an $O(m)$--topological functor from $\jcal_m$ to based $O(m)$--spaces.
\end{definition}

When discussing the composition of restriction and evaluation, we sometimes omit notation for restriction.
On a vector space $V$, $(X \circ i_m^n )/O(n-m)(V) = X(V)/O(n-m)$, which is
an $O(m)$--space. These restriction functors both have right adjoints. The first step is to identify the right adjoint of the orbit functor.

\begin{lemma}
There is an adjoint pair
$$(-)/O(n-m) : O(n) \Top
\overrightarrow{\longleftarrow}
O(m) \Top : \CI_m^n $$
\end{lemma}
The right adjoint is defined as the composite of two functors. The first takes an $O(m)$--space $A$ and considers it as an $O(m) \times O(n-m)$--space by letting the $O(n-m)$--factor act trivially, this is called $\varepsilon^* A$. The second functor takes $\varepsilon^* A$ and sends it to the topological space of $O(m) \times O(n-m)$--maps from $O(n)$ to $\varepsilon^* A$.
We can therefore write 
\[
\CI_m^n A = F_{O(m) \times O(n-m)}(O(n)_+, \varepsilon^* A)
\]

\begin{lemma}
There is a right adjoint to $res_m^n$,
called \textbf{induction}, defined as
$$(\ind_m^n X)(V) = \nat_{\ecal_m} (\mor_n(V,-), X)$$
where the right hand side is the topological space of maps between two objects of $\ecal_m$.

There is a right adjoint to $res_m^n/O(n-m)$,
called \textbf{inflation--induction}, which we write as $\ind_m^n \CI X$,
it is defined as
$$(\ind_m^n \CI X)(V) = \nat_{O(m)\ecal_m} (\mor_n(V,-), \CI_m^n X)$$
\end{lemma}
When $m=0$ we usually replace $\CI_0^n$ with $\varepsilon^*$, as here $\CI_0^n$ is simply equipping $X$ with the trivial $O(n)$--action.

Now we may define differentiation. The motivation for this
definition comes from two lemmas. 
Firstly, Lemma \ref{lem:stabletoinduction}, which can be thought 
of as describing differentiation
as a measure of a `rate of change'. 
Secondly, Lemma \ref{lem:inductionandtau}, which can be thought 
of as describing differentiation as a measure of how far 
a functor is from being `polynomial'.

\begin{definition}
Let $E \in \ecal_0$, then the \textbf{$n^{th}$--derivative} of $E$ is
$\ind_0^n \varepsilon^* E \in O(n) \ecal_n$.
\end{definition}
For $F \in \ecal_n$ we also talk of $\ind_n^{n+1} F$ 
as being the derivative of $F$. 

Note that $i^* \ind_0^n \varepsilon^* E \in \ecal_n$
is equal to $\ind_0^n E$, so $\ind_0^n E(V)$ 
has an $O(n)$--action for each $V$.
Indeed, \cite[Proposition 3.1]{weiss95} 
uses this fact to say (using our new language) that 
$\ind_0^n E$ can be thought of as an object of  $O(n) \ecal_n$.
That object is precisely $\ind_0^n \varepsilon^* E$. 

We are most interested in the pair $(\res_0^n/O(n), \ind_0^n \varepsilon^*)$,
though we will need the non-equivariant functor
$\ind_m^n$ for some calculations. We will sometimes omit $\res_0^n$ from our notation, provided that no confusion can occur. 

We can give another relation between induction and 
the categories $\jcal_n$ for varying $n$. 
The following is \cite[Proposition 1.2]{weiss95}, which shows how one can construct
$\jcal_{n+1}$ from $\jcal_n$.
\begin{proposition}\label{prop:jntojn}
For all $V$ and $W$ in $\jcal_0$
and all $n \geqslant 0$ there is a
natural homotopy cofibre sequence
$$
\mor_n(\rr \oplus V, W) \smashprod S^{n}
\to
\mor_n (V,W)
\to
\mor_{n+1} (V,W)
$$
\end{proposition}

\begin{proof}
Identifying $S^n$ as the closure of the subspace $(i,x) \in \gamma_n (V,\rr \oplus V)$, where $i$ is the standard inclusion, the composition map 
\[
\mor_n (\rr \oplus V, W) \smashprod \mor_n (V, \rr \oplus V) \to \mor_n (V,W)
\] 
restricts to a morphism 
$\mor_n(\rr \oplus V, W) \smashprod S^{n}\to \mor_n (V,W)$.  
The homotopy cofibre of the restriction is then a quotient of $[ 0, \infty ] \times \gamma_n (\rr \oplus V,W) \times \rr^n$.  The desired homeomorphism, away from the base point, is induced by the association below. 
Consider a quadruple
$$
(t \in \lbrack 0, \infty \rbrack,
f \in 
\rr \oplus V, W) ,
y \in \rr^n \otimes (W - f(\rr \oplus V)),
z \in \rr^n)
$$
we send this to the element $(f|_V ,x) \in \mor_{n+1}(V,W)$. 
Where $x= y+ (f|_{\rr} *) (z)+ t \omega( f|_{\rr *} (1))$, 
and $\omega : W \to \rr^{n+1} \otimes W$ identifies 
$W \cong (\rr^n \otimes W)^{\bot}\subset \rr^{n+1} \otimes W $.
\end{proof}

From this cofibre sequence we can make a fibre sequence by applying the
functor $\nat_{\ecal_n}(-,F)$ for $F \in \ecal_n$. The following result is \cite[Proposition 2.2]{weiss95}.
\begin{lemma}\label{lem:stabletoinduction}
For all $V \in \jcal_n$ and  $F \in \ecal_n$, there is a natural homotopy fibre sequence
\[
\res_n^{n+1} \ind_n^{n+1} F(V) 
\longrightarrow F(V) 
\longrightarrow
\Omega^n F(\rr \oplus V)
\]
\end{lemma}

\section{\texorpdfstring{$n$}{n}-polynomial functors}\label{sec:npolyfun}

We want to study a well-behaved collection of functors in $\ecal_0$: those whose
derivatives are eventually trivial. By analogy with functions on the real numbers, we call these functors polynomial. In this section we introduce this class of functors and examine
how they relate to differentiation.

\begin{definition}
For vector spaces $V$ and $W$ in $\jcal_0$, let $S\gamma_{n+1}(V,W)$ be the total space of the unit sphere vector bundle of $\gamma_{n+1}(V,W)$. 
\end{definition}
We can think of
$S\gamma_{n+1}(-,-)_+$ as a continuous functor from
$\jcal_0^{op} \times \jcal_0$ to based spaces, we can use this to define a functor
from $\ecal_0$ to itself. 

\begin{definition}
For $E \in \ecal_0$, define $\tau_n E \in \ecal_0$ by
$$(\tau_n E)(V) = \nat_{\ecal_0}(S \gamma_{n+1}(V,-)_+, E)$$
We also have a natural transformation of self-functors
on $\ecal_0$:
$$\rho_n \co \id \to \tau_n $$
\end{definition}
This natural transformation comes from the map
$S\gamma_{n+1}(V,W)_+ \to \mor_0(V,W)$
and the Yoneda lemma. 

There is another description of $S\gamma_{n+1}(-,-)$, by
\cite[Proposition 4.2]{weiss95} it is a homotopy colimit:
$$S\gamma_{n+1}(V,A)_+ \cong
\underset{0 \neq U \subset \rr^{n+1}}{\hocolim} \mor_0 (U \oplus V, A)$$
where the right hand side is the Bousfield-Kan formula for the homotopy colimit of 
the functor $U \to \mor_0(U \oplus V)$ as
$U$ varies over the topological category of 
non-zero subspaces of $\rr^{n+1}$ and inclusions. 
This construction is described in great detail in 
\cite[Appendix A]{lind}. 
Thus we see that
\[
\tau_n E(V) = \underset{0 \neq U \subset \rr^{n+1}}{\holim} 
E(U \oplus V)
\]

We choose to define $\tau_n$ 
in terms of $S\gamma_{n+1}(-,-)_+$
and we then define 
polynomial functors in terms of $\tau_n$.
Thus the definition below is \cite[Proposition 5.2]{weiss95}.

\begin{definition}\label{def:toppoly}
A functor $E$ from $\jcal_0$ to based spaces
is said to be \textbf{polynomial of degree less than or equal to $n$} if and only if
$$
(\rho_n)_E \co E \to \tau_n E
$$
is an objectwise weak equivalence of $\jcal_0$--spaces.
\end{definition}

We sometimes say that such an $E$ is \textbf{$n$--polynomial}.
The value of an $n$--polynomial
functor $E$ at $V$ is determined, up to homotopy, 
by the values $E(U \oplus V)$ (and the maps between them)
for non-zero subspaces of $U$ of $\rr^{n+1}$. Hence we can think 
of an $n$--polynomial functor as one where it is possible to 
extrapolate the information
of $E(U)$ from the spaces $E(U \oplus V)$ (and maps between them).

The homotopy fibre of $\rho_n \co E \to \tau_n E$ measures 
how far $E$ is from being $n$--polynomial, 
thus it would be helpful to be able to identify this fibre.
The following lemma, from \cite[Section 5]{weiss95}, does so and 
shows the fundamental relation between differentiation and 
$n$--polynomial functors.

\begin{proposition}\label{prop:sgammatojn}
The topological space $\mor_{n+1}(V,A)$
is the mapping cone (cofibre) of the projection
$S\gamma_{n+1}(V,A)_+  \to \mor_0(V,A)$.
This statement is natural in $V$ and $A$.
\end{proposition}
\begin{proof}
The mapping cone is the pushout of the diagram below,
where we use $[0,\infty] = [0, \infty)^c$ (with basepoint $\infty$)
instead of the unit interval. This helps in identifying the pushout. 
$$\xymatrix{
S\gamma_{n+1}(V,A)_+  \ar[r] \ar[d] & \mor_0(V,A)_+ \ar[d] \\
S\gamma_{n+1}(V,A)_+ \smashprod [0, \infty] \ar[r] & P
}$$
The top horizontal map is the projection, the left vertical map
sends a point $x$ to $(x,0)$.

The idea is that every element of $V$ can be written as a unit 
vector times some length :
$S(V) \times [0, \infty) \cong V$.
Thus writing $S^V$ for the one-point compactification of $V$, we see that 
$S(V) \times [0, \infty] \cong S^V$,
where any vector of 'infinite length'
is identified with the point at infinity in $S^V$.

The pushout consists of points
$(f,x,t)$, where $t \in [0, \infty]$ and
$(f,x) \in S\gamma_{n+1}(V,A)$, modulo the relations
$(f,x,\infty) = (f',x',\infty)$
and $(f,x,0) = (f,x',0)$.

We have a map from this pushout to $\mor_{n+1} (V,W)$,
it sends any point of form $(f,x,\infty)$ to the basepoint and sends
$(f,x,t)$ to $(f,xt)$ for all other $t$. It is clear that this is a well-defined map;
indeed, it is a homeomorphism.
\end{proof}

\begin{lemma}\label{lem:inductionandtau}
For any $n \in \nn$, $V \in \jcal_0$ and $E \in \ecal_0$, there exists a natural fibration sequence
$$
\res^{n+1}_0 \ind^{n+1}_0 E(V) \to E(V) \to \tau_n E(V)
$$
\end{lemma}

\begin{proof}
We have the natural cofibre sequence
$$S\gamma_{n+1}(V,A)_+  \to \mor_0(V,A)_+ \to \mor_{n+1}(V,A)$$
which is natural in $V$ and $A$ with respect to $\jcal_0$.
This assembles to give a cofibre sequence
of $\jcal_0$--spaces:
$$S\gamma_{n+1}(V,-)_+  \to \mor_0(V,-) \to \mor_{n+1}(V,-)$$
Now consider the induced maps of spaces
$$\nat_{\ecal_0}( S\gamma_{n+1}(V,-)_+,E)
\leftarrow
\nat_{\ecal_0}(\mor_0(V,-),E)
\leftarrow
\nat_{\ecal_0}(\mor_{n+1}(V,-),E)$$
We can identify the above with
$$(\tau_n E)(V)  \leftarrow E(V) \leftarrow 
(\res_0^n \ind_0^{n+1}  E)(V)$$
which is a fibre sequence for all $V$.
\end{proof}

\begin{corollary}\label{cor:npolydiff}
Let $E$ a functor from $\jcal_0$ to based spaces
that is $n$--polynomial.
Then $\ind_0^{n+1} E$ (and hence 
$\ind_0^{n+1}  \varepsilon^* E$)  
is objectwise acyclic.
\end{corollary}

As one would hope from the words used, 
any $(n-1)$--polynomial object of $\ecal_0$
is $n$--polynomial. That result is 
\cite[Proposition 5.4]{weiss95}, which we reproduce 
later as proposition \ref{prop:mpolynpoly}.

Our goal is to construct a tower relating the $n$ and $(n-1)$--polynomial approximations of an object $E$ of $\ecal_0$ and  classify the fibres of this tower. Any such fibre will be $n$--polynomial and be $T_{n-1}$--contractible, hence we make the following definition, see
\cite[Definition 7.1]{weiss95}.

\begin{definition}\label{def:tophomog}
An object $E \in \ecal_0$ is \textbf{$n$--homogeneous} if it is
polynomial of degree at most $n$ and $T_{n-1} E$ is
weakly equivalent to a point.  
\end{definition}

We will construct model structures that capture the notion
of $n$--polynomial or $n$--homogeneous functors in their homotopy
theory. To do so, we will need some more technical
information on $n$--polynomial functors.

A routine exercise in using the long exact 
sequence of a fibration and the five lemma
gives \cite[Lemma 5.3]{weiss95}, which is 
stated below. 
\begin{lemma}\label{lem:fibreisnpoly}
Let $g \co E \to F$ be a map in $\ecal_0$, assume that 
$\ind_0^{n+1} F$ is objectwise contractible
and $E$ is $n$--polynomial. 
Then the homotopy fibre of $g$ is an 
$n$--polynomial functor. 
\end{lemma}

In particular, this proves that the homotopy fibre of
a map between $n$--polynomial objects is $n$--polynomial. 
We now need \cite[Definition 5.9]{weiss95}, 
this condition often crops up. The following lemma
is an example of why this notion is useful. 

\begin{definition}
We say that a functor $E \in \ecal_0$
is \textbf{connected at infinity}
if the space $\hocolim_k E(\rr^k)$
is connected. 
\end{definition}

The following result is \cite[Proposition 5.10]{weiss95}
and we follow that proof. 

\begin{lemma}\label{lem:npolyfibre} 
Let $g \co E \to F$ be a morphism in $\ecal_0$
between $n$--polynomial objects
such that the homotopy fibre of $g$ is objectwise
acyclic and $F$ is connected at infinity. 
Then $g$ is an objectwise weak equivalence.
\end{lemma}
\begin{proof}
The problem lies in the fact that at each stage $V$, 
the homotopy
fibre is defined via a fixed choice of basepoint in $F(V)$, 
but we need an isomorphism of homotopy groups 
between $E(V)$ and $F(V)$ for all choices of basepoints. 
Let $F_b(V)$ be the subspace of $F(V)$ consisting 
of only the basepoint component of $F(V)$.

We prove that $F_b \to F$ is a
equivalence after applying the functor 
$T_n = \hocolim_k \tau_n^k$.  
Note that since $E$ and $F$ are $n$--polynomial, 
the maps $E \to T_n E$ and $F \to T_n F$ are objectwise weak equivalences. 
Consider the map
\[
\hocolim_k \tau_n^k F_b 
\longrightarrow
\hocolim_k \tau_n^k F 
\]
For each choice of basepoint, the homotopy fibre of 
$\tau_n^k F_b \to \tau_n^k F$
is either empty or contractible.
If $C$ is some component in  
$F(V) \simeq \tau_n^k F(V)$, 
then because $f$ is connected at infinity,
there is some $l$ such that 
the image of $C$ in 
$\tau_n^l F(V)$ is in the basepoint component. 
This holds since 
$\tau_n^l F(V)$ is defined using only the terms
$F(V \oplus U)$ for $U$ of dimension greater
than or equal to $l$. 
Hence $C$ is contained in 
$\tau_n^l F_b(V)$ and there can be no empty fibres.

We thus have objectwise weak equivalences
\[
T_n F_b \to T_n F
\]
Consider the map $T_n E(V) \to T_n F(V)$
and choose some basepoint $x$ in $T_n F(V)$, 
then we see that $x \in \tau_n^k F(V)$ for some $k$.
As $k$ increases, eventually $x$ is in the same component
as the canonical basepoint of $\tau_n^k F(V)$.
Hence by our assumptions, the homotopy fibre for this choice 
$x$ is contractible. 
So $T_n E \to T_n F$ is a objectwise weak equivalence
and it follows that $E \to F$ is a objectwise weak equivalence. 
\end{proof}

Now we turn to \cite[lemma 5.11]{weiss95}
to show that for $\tau_m$ preserves 
$n$--polynomial functors. The proof is simply 
that homotopy limits commute, 
(so $\tau_n \tau_m = \tau_m \tau_n$)
and that homotopy limits preserve
weak equivalences.   
\begin{lemma}\label{lem:taumEisnpoly}
If $E$ is an $n$--polynomial object
of $\ecal_0$, then so is $\tau_m E$ for any $m \geqslant 0$. 
\end{lemma}

We will need an improved 
version of corollary \ref{cor:npolydiff}, specifically, we need
\cite[Corollary 5.12]{weiss95}, which we give below.
We will see later that this result 
implies that 
$\ind_0^n \varepsilon^*$ takes fibrant objects of the 
$n$--polynomial model structure on $\ecal_0$ to 
fibrant objects of the $n$--stable model structure on 
$O(n) \ecal_n$. 

\begin{proposition}\label{prop:npolynstable}
If $E$ is an $n$--polynomial object of $\ecal_0$, 
then for any $V \in \jcal_0$, we have a weak equivalence of 
spaces 
\[
\ind_0^n E(V) \to \Omega^n \ind_0^n E(V \oplus \rr)
\] 
\end{proposition}
\begin{proof}
When $n=0$ there is nothing to prove, so let $n >0$. 
By corollary \ref{cor:npolydiff} we see that 
$\ind_0^{n+1} E$ is objectwise contractible, 
but this space also appears 
in the homotopy fibration sequence of 
lemma \ref{lem:stabletoinduction}.
\[
\res_n^{n+1} \ind_0^{n+1} E(V)
\longrightarrow
\ind_0^n E(V)
\longrightarrow
\Omega^n \ind_0^n E(V \oplus \rr)
\]
We claim that 
$\ind_0^n E$ and the functor 
$F$ defined by the rule 
$V \mapsto \Omega^n \ind_0^n E(V \oplus \rr)$ are both $n$--polynomial. 
Furthermore we claim that $F$ is 
connected at infinity. 

Once we have shown this, we will be able to apply 
lemma \ref{lem:npolyfibre} to see that 
we have a weak equivalence for all $V \in \jcal_0$
\[
\ind_0^n E(V)
\longrightarrow
\Omega^n \ind_0^n E(V \oplus \rr)
\]

Our first claim was that $\ind_0^n E$ 
is an $n$--polynomial object of $\ecal_0$. 
We know that $E$ is $n$--polynomial
and lemma \ref{lem:taumEisnpoly}
tells us that $\tau_{n-1} E$ is $n$--polynomial. 
The homotopy fibre of 
$E \to \tau_{n-1} E$ is 
$\ind_0^n E$ and by 
by lemma \ref{lem:fibreisnpoly} we see that it is 
$n$--polynomial.

Our second claim was that 
$F$ was $n$--polynomial. 
We know that the functor
$\ind_0^n E$ is $n$--polynomial
by the first claim, $\Omega$
preserves $n$--polynomial functors, 
thus $\Omega^n \ind_0^n E$ is $n$--polynomial.
Now we use the fact that if $A$ is $n$--polynomial, 
then the functor $V \mapsto A(V \oplus \rr)$
is $n$--polynomial, and thus $F$ is $n$--polynomial. 

Our third claim was that 
$F$ is connected at infinity, this follows
since it is the restriction of 
an object of $\ecal_n$.
\end{proof}

\section{Polynomial and homogeneous model structures}\label{sec:npolymodel}

As with calculus in the smooth setting, we wish to approximate a functor in $\ecal_0$ by an $n$--polynomial functor. This is done by iterating $\tau_n$ to construct a functorial $n$--polynomial replacement. From this we can create a new model structure on $\ecal_0$, where the fibrant objects are $n$--polynomial. We show that this model structure
can also be created using a left Bousfield localisation. Combining these two methods tells us more about the $n$--polynomial model structure, in particular, we see that it is right proper and cellular and hence can undergo a right Bousfield localisation.  By a careful choice of localisation we construct a model category whose cofibrant-fibrant
objects are precisely the $n$--homogeneous functors of $\ecal_0$. 
We start by introducing the projective model structure model structure on $\ecal_0$.

When equipped with this model structure, $\ecal_0$ is a topological model category, in the sense of \cite[Definition 4.2.18]{hov99}. 
Hence we see that the enrichment, tensor product and cotensor product
are well behaved with respect to the model structure. 
This is analogous to simplicial model categories, but with topological
spaces taking the place of simplicial sets. 

Recall the notion of cellular model categories from \cite[Section 12]{hir03}, this is a stronger version of cofibrant generation. For example, based topological spaces
are cellular. This condition is a requirement for our localisations to exist. 
The following result is \cite[Theorem 6.5]{mmss01} combined with a  
routine, but technical, argument to see that the resulting model structure is cellular. 

\begin{lemma}
There is a proper, cellular model structure on the category $\ecal_0$ where the fibrations and weak equivalences are defined objectwise. This is known as the \textbf{projective model structure} and we simply write $\ecal_0$ for this model structure. The generating cofibrations have form 
$$\mor_0(V,-) \smashprod S^{n-1}_+ \to \mor_0(V,-) \smashprod D^n_+$$ 
and the generating acyclic cofibrations have form
$$\mor_0(V,-) \smashprod D^{n}_+ \to \mor_0(V,-) \smashprod (D^n \times [0,1])_+$$ 
for $V \in \jcal_0$ and $n \geqslant 0$.
Let $[-,-]$ denote maps in the homotopy category of 
$\ecal_0$.
\end{lemma}

\begin{lemma}\label{lem:coffunct}
The functors $S \gamma_{n}(V,-)_+$ and $\mor_n (V,-)$ are cofibrant
objects of $\ecal_0$, for $n \geqslant 0$. 
\end{lemma}
\begin{proof}
The homotopy limit used to construct 
$\tau_n$ preserves objectwise
fibrations and acyclic fibrations by 
\cite[Theorem 18.5.1]{hir03}.  
It follows that $S \gamma_{n}(V,-)_+$
is cofibrant. Since $\mor_{n+1}(V,-)$ is the mapping cone
of a map between two cofibrant objects,
$S \gamma_{n}(V,-)_+ \to \mor_0(V,-)$, 
it is also cofibrant. 
\end{proof}

\begin{definition}\label{def:tn}
Define $T_n \co \ecal_0 \to \ecal_0$ to be
$$
T_n E = \hocolim
\xymatrix{
E \ar[r]^{\rho_E} &
\tau_n E \ar[r]^{\tau_n \rho_E} &
\tau_n^2 E \ar[r]^{\tau_n^2 \rho_E} &
\dots }$$
The inclusion map
$(\eta_n)_E \co E \to T_n E$
is a natural transformation.
\end{definition}

\begin{definition}
A map $f \in \ecal_0$ is said to be an \textbf{$T_n$--equivalence}
if $T_n f$ is an objectwise weak equivalence.
\end{definition}

We use the functor $T_n: \ecal_0 \to \ecal_0$ to construct a new model structure
on $\ecal_0$. The method is known as Bousfield-Friedlander localisation with respect to $T_n$. Specifically we apply \cite[Theorem 9.3]{bous01} which is an updated and improved version of \cite[Theorem A.7]{bf78}. We will shortly obtain this model structure via a different method, 
but we need this version to see that the new model structure is right proper. 

\begin{proposition}\label{prop:bflocal}
There exists a proper model structure
on $\ecal_0$ such that a map $f$ is a weak equivalence
if and only if $f$ is an $T_n$--equivalence.
The cofibrations are the cofibrations of the projective model structure on $\ecal_0$. The
fibrant objects are precisely the $n$--polynomial objects. A map
$f \co X \to Y$ is a fibration if and only if it is an objectwise fibration and
the diagram below is a homotopy pullback in the projective model structure.
$$\xymatrix{
X \ar[r]^f \ar[d]^\rho & Y \ar[d]^\rho \\
T_n X \ar[r]^{T_n f} \ar[r] & T_n Y
}$$
We call this the \textbf{$n$--polynomial model structure} on $\ecal_0$
and denote it by $n\poly \ecal_0$.
\end{proposition}
\begin{proof}
We need to show the following axioms:
\begin{itemize}
\item (A1) if $f:X \rightarrow Y$ is an objectwise weak equivalence, then so is $T_n f$;
\item (A2) for each $X \in \ecal_0$, the maps $\eta, T_n \eta : T_n X \rightarrow T^{2}_{n} X$ are weak equivalences;
\item (A3) for a pull back square
$$\xymatrix{
V \ar[r]^k \ar[d]^g & X \ar[d]^f \\
W \ar[r]^{h} \ar[r] &Y
}$$
in $\ecal_0$, if f is a fibration of fibrant objects such that $\eta : X\rightarrow T_n X, \eta : Y \rightarrow T_n Y$, and $ T_n h: T_n W \rightarrow T_n Y$ are weak equivalences, then $T_n k : T_n V \rightarrow T_n X$ is a weak equivalence.
\end{itemize}
Weiss proves axioms A1 and A2 in \cite[Theorem 6.3]{weiss95} 
and \cite{weisserrata}. For A3, recall that finite homotopy limits commute with directed homotopy colimits, and therefore $T_n$ preserves finite homotopy limits and A3 follows.
\end{proof}

It follows immediately that the fibrations of $n \poly \ecal_0$ are precisely those objectwise fibrations that are also $n$--polynomial maps, as defined in \cite[Definition 8.1]{weiss95}. Similarly an $n$--homogeneous map is an $n$--polynomial
map that is a weak equivalence in $(n-1) \poly \ecal_0$. Thus we are justified in
saying that these model categories contain the homotopy theory that Weiss studies.

This new model structure turns the identity functor into
the left adjoint of a Quillen pair from the projective
model structure on $\ecal_0$ to the $n$--polynomial model structure.
\[
\id :
\ecal_0 \overrightarrow{\longleftarrow}
n \poly \ecal_0
: \id
\]
If we let $[-,-]^{np}$ denote maps in the homotopy category of 
$n \poly \ecal_0$ then we see that for $X$ and $Y$ in 
$\ecal_0$
\[
[X,T_n Y] \cong
[X,Y]^{np}
\]

There is another way to obtain the $n$--polynomial model structure, here we use the left Bousfield localisations of \cite{hir03}.  This technique is more modern 
than that of the $T_n$--localisation above. It has the major advantage that we can now
conclude that the model category $n \poly \ecal_0$ is cellular, which we will need shortly. 

\begin{proposition}\label{prop:cellular}
The model category $n \poly \ecal_0$ is 
the left Bousfield localisation of $\ecal_0$ with respect
to the collection below and hence is cellular and topological.
\[
S_n = \{ S\gamma_{n+1}(V,-)_+ \to \mor_0(V,-) | V \in \jcal_0
\}
\]
\end{proposition}
\begin{proof}
We know that $\ecal_0$ is left proper and cellular, 
hence \cite[Theorem 4.1.1]{hir03} applies and we see that
the localisation $L_{S_n} \ecal_0$ exists and is left proper, cellular and topological. 

The cofibrations are those from $\ecal_0$, the fibrant objects and 
weak equivalences are defined in terms of homotopy mapping objects 
as we describe below. Note that a weak equivalence between 
fibrant objects is precisely an objectwise weak equivalence.

If we can show that $n \poly \ecal_0$ has the same weak equivalences
as $L_{S_n} \ecal_0$, then we will know that these model structures are precisely the 
same and hence cellular and proper. 

Since $\ecal_0$ is enriched over topological spaces and all objects are fibrant,
the homotopy mapping object from $A$ to $B$
is given by the enrichment $\nat_{\ecal_0}(\cofrep A, B)$
where $\cofrep$ denotes cofibrant replacement. 
See \cite[Example 17.2.4]{hir03} for more details.

The domains and codomains of $S_n$ are cofibrant by lemma \ref{lem:coffunct}, hence 
the fibrant objects of $L_{S_n} \ecal_0$ are those $X$
such that 
\[
(\rho_n)_X(V) X(V) = \nat_{\ecal_0} (\mor_0(V,-),X) \to 
\nat_{\ecal_0} (S\gamma_{n+1}(V,-)_+,X) = (\tau_n X)(V)
\]
is a weak homotopy equivalence of based spaces for all $V$. 
So we see that the 
fibrant objects are precisely the $n$--polynomial objects. 
The weak equivalences of $L_{S_n} \ecal_0$ are those maps 
$f \co X \to Y$ such that the induced 
\[
\nat_{\ecal_0} (\cofrep Y,Z) \to 
\nat_{\ecal_0} (\cofrep X,Z)
\]
gives a weak homotopy equivalence of spaces 
whenever $Z$ is $n$--polynomial. 
Recall that $[-,-]$ denotes maps in the homotopy category of 
$\ecal_0$ and 
$[-,-]^{np}$ denotes maps in the homotopy category of 
$n \poly \ecal_0$. 
Since $\ecal_0$ is a topological model category 
we can relate the homotopy groups of 
$\nat_{\ecal_0} (-,-)$ to maps in the homotopy category of
$\ecal_0$:
\[
\pi_n \nat_{\ecal_0} (\cofrep X,Z)
\cong
[X,\Omega^n Z]
\cong
[X,\Omega^n Z]^{np}
\]
where $Z$ is $n$--polynomial. 
Since $X \to T_n X$ is a $T_n$ equivalence
it follows that 
\[
\pi_* \nat_{\ecal_0} (\cofrep X,Z)
\cong 
\pi_* \nat_{\ecal_0} (\cofrep T_n X,Z)
\]
Hence the map $\eta \co X \to T_n X$ is a weak equivalence of $L_{S_n} \ecal_0$
and so the collection of $S_n$--equivalences is precisely the class
of $T_n$--equivalences. 
\end{proof}

Now that we know the $S_n$--equivalences
are the weak equivalences for the $n$--polynomial
model structure, we can give a proof that 
an $(n-1)$--polynomial object is $n$--polynomial. 
The following result is a reproduction of 
\cite[Proposition 5.4]{weiss95}. 
We note that 
the `$E$--substitutions' of the reference
are precisely $S_{n-1}$--equivalences and 
hence are $T_{n-1}$--equivalences. 

\begin{proposition}\label{prop:mpolynpoly}
If $E \in \ecal_0$ is polynomial of degree at most $n-1$,
then it is polynomial of degree at most $n$.
\end{proposition}
\begin{proof}
What we will actually show is that 
any $S_{n}$--equivalence is an $S_{n-1}$--equivalence.
Thus we must prove that 
$S\gamma_{n+1}(V,W) \to \mor_0(V,-)$
is an $S_{n-1}$--equivalence for any $V$.  
By the two-out-of-three property
we can reduce this to proving that the map $\alpha$ below is an
$S_{n-1}$--equivalence. 

The standard inclusion $\rr^n \to \rr^{n+1}$
induces a map of vector bundles
$\gamma_{n}(V,W) \to \gamma_{n+1}(V,W)$
and hence a map of their respective
unit sphere bundles. 
\[
\alpha \co S\gamma_{n}(V,-)_+  \to S \gamma_{n+1}(V,-)_+
\] 
We can write $S \gamma_{n+1}(V,-)_+$ as the fibrewise product
over $\mor_0(V,-)$ (denoted $\boxtimes$) of 
$S \gamma_{n}(V,-)_+$ and 
$S \gamma_{1}(V,-)_+$. 
Thus we can write 
$S \gamma_{n+1}(V,-)_+$ as the homotopy
pushout of the diagram
\[
S \gamma_{n}(V,-)_+
\overset{p_1}{\longleftarrow}
S \gamma_{n}(V,-)_+
\boxtimes
S \gamma_{1}(V,-)_+  
\overset{p_2}{\longrightarrow} 
S \gamma_{1}(V,-)_+
\]
Writing 
$\epsilon^n$ for the $n$--dimensional
trivial bundle, 
we see that there is a pullback square
\[
\xymatrix{
(\epsilon^n \oplus \gamma_n (\rr \oplus V, -) )_+ 
\ar[r] \ar[d] &
\gamma_{n}(V,-)_+ \ar[d] \\
\mor_0(\rr \oplus V, -) \ar[r] & 
\mor_0(V,-)
}
\]
The map $p_2$ can then be identified as 
the projection map 
\[
S(\epsilon^n \oplus \gamma_n (\rr \oplus V, -) )_+
\longrightarrow
\mor_0(\rr \oplus V, -)
\]
Hence the vector bundle 
$S \gamma_{n+1}(V,-)_+$ is the homotopy pushout of 
\[
S \gamma_{n}(V,-)_+
\longleftarrow
S( \epsilon^n \oplus \gamma_n (\rr \oplus V, -))_+  
\overset{p_2}{\longrightarrow} 
\mor_0(\rr \oplus V, -)
\]
If $p_2$ is an $S_{n-1}$--equivalence,
then so is its homotopy pushout, which is $\alpha$. 
The unit sphere of a Whitney sum of vector bundles
is equal to the fibrewise join of the unit sphere bundles.
Hence the domain of $p_2$
can be written as the homotopy pushout  
\[
S \gamma_{n}(\rr \oplus V,-)_+
\longleftarrow
S^{n-1}_+ \smashprod 
S \gamma_{n}(\rr \oplus V,-)_+ 
\overset{\delta}{\longrightarrow} 
S^{n-1}_+ \smashprod 
\mor_0(\rr \oplus V,-)
\]
The map $\delta$ is an $S_{n-1}$--equivalence, hence
the top map in the commutative diagram 
below is an $S_{n-1}$--equivalence.
\[
\xymatrix{
S \gamma_{n}(\rr \oplus V,-)_+ 
\ar[r] \ar[dr] &
S( \epsilon^n \oplus \gamma_n (\rr \oplus V, -)_+ 
\ar[d]^{p_2} \\
& 
\mor_0(\rr \oplus V,-)
}
\]
Since the diagonal map is an element of $S_{n-1}$, 
it follows that $p_2$ is an $S_{n-1}$--equivalence, as desired. 
\end{proof}

\begin{corollary}
The identity functor is the left adjoint of a Quillen pair from the $n$--polynomial
model structure to the $m$--polynomial model structure, for $m < n$.
\[
\id :
n \poly \ecal_0 \overrightarrow{\longleftarrow}
m \poly \ecal_0
: \id
\]
\end{corollary}

To define a homotopy theory for $n$--homogeneous functors, we construct a right Bousfield localisation of the $n$--polynomial structure.

\begin{proposition}
There is a topological model structure on $\ecal_0$
whose cofibrant-fibrant objects are precisely the class of 
$n$--homogeneous objects that are cofibrant
in the projective model structure in $\ecal_0$. 
We denote this model structure
$n \homog \ecal_0$. There is a Quillen pair
\[
\id :
n \homog \ecal_0 
\overrightarrow{\longleftarrow}
n \poly \ecal_0
: \id
\]
The fibrations of $n \homog \ecal_0$
are the same as those for 
$n \poly \ecal_0$.
The weak equivalences of 
$n \homog \ecal_0$ are those maps $f$
such that $\res_0^n \ind_0^n T_n f$ is an
objectwise weak equivalence in 
$\ecal_0$.  
\end{proposition}
\begin{proof}
We produce this model category by a right Bousfield localisation 
of $n \poly \ecal_0$ at the 
collection of objects 
\[
K_n= \{ \mor_n(V,-) | V \in \jcal_0 \}
\]
Since $n \poly \ecal_0$ is right proper and cellular, we are able to use 
\cite[Theorem 5.1.1]{hir03} to see that this localisation, 
$R_{K_n}( n \poly \ecal_0)$, exists.

The weak equivalences of this model category
are called $K_n$--cellular equivalences and are 
defined in terms of homotopy mapping objects. 
By \cite[Example 17.2.4]{hir03} we see that 
if $A$ is cofibrant in $\ecal_0$ then 
the homotopy mapping object from $A$ to $X$ is given by 
$\nat_{\ecal_0}(A, T_n X)$.
This is different from the homotopy mapping object used previously, 
as we now want a homotopy mapping object for $n \poly \ecal_0$. 

A map $f \co X \to Y$ is a \textbf{$K_n$--cellular equivalence} in 
$R_{K_n} ( n \poly \ecal_0)$ if and only if 
\[
\xymatrix{
(\ind_0^n T_n X)(V) \ar[r]^(0.4)= &
\nat_{\ecal_0}(\mor_n(V,-), T_n X)
\ar[d]^{T_n f_*} \\ 
(\ind_0^n T_n Y)(V) \ar[r]^(0.4)= &
\nat_{\ecal_0}(\mor_n(V,-), T_n Y) 
}
\]
is a weak equivalence for all $V$. 
Hence we have proven that 
the weak equivalences of this new model structure 
are those maps $f$ such that $\res_0^n \ind_0^n T_n f$
is an objectwise weak equivalence of $\ecal_0$.
Furthermore, any $(n-1)$--polynomial
object is trivial in this new model structure. 

An object $A$ is cofibrant in $R_{K_n} ( n \poly \ecal_0)$
if and only if it is cofibrant in $\ecal_0$ and 
for any $K_n$--cellular equivalence $X \to Y$, the induced map 
\[
\nat_{\ecal_0}(A, T_n X) 
\longrightarrow
\nat_{\ecal_0}(A, T_n Y)
\]
is a weak equivalence of spaces. 
The general theory of right localisations tell us that 
a $K_n$--cellular equivalence between cofibrant objects is a 
$T_n$--equivalence.

Now we show that the 
cofibrant--fibrant objects are
$n$--homogeneous.
Let $A$ be cofibrant and fibrant in 
$R_{K_n} ( n \poly \ecal_0)$. Then 
$\ast \to T_{n-1} A$ is a 
$K_n$--cellular equivalence, 
 and hence we have a weak equivalence of spaces
\[
\nat_{\ecal_0}(A, \ast) 
\longrightarrow
\nat_{\ecal_0}(A, T_{n-1} A)
\]
Thus we have isomorphisms
\[
0 =
[A, T_{n-1} A]
\cong 
[A, T_{n-1} A]^{(n-1)p}
\cong
[T_{n-1} A, T_{n-1} A]^{(n-1)p}
\cong
[T_{n-1} A, T_{n-1} A]
\]
This tells us that $T_{n-1} A$ is objectwise contractible. 
Since $A$ is fibrant, we know it is $n$--polynomial
and we now know that $A$ is $n$--homogeneous.

Now we must show that if $B$ is a cofibrant 
object of $\ecal_0$ which is $n$--homogeneous,
then it is cofibrant and fibrant in 
$n \homog \ecal_0$. We immediately see 
that $B$ is fibrant, as it is $n$--polynomial.
So consider the cofibrant replacement
of $B$ in the $n$--homogeneous model structure
$f \co \cofrep B \to B$. Since the codomain is 
fibrant, so is the domain and thus we know that 
$\res_0^n \ind_0^n f$ is an objectwise weak equivalence. 
We must prove that $f$ itself is an objectwise
weak equivalence. We do so by adapting some of 
the argument of \cite[Corollary 5.13]{weiss95}.

The homotopy fibre of $f$, which we call $D$, 
is $n$--polynomial and $\res_0^n \ind_0^n D$ is 
objectwise contractible.
Consequently, the map 
$D \to \tau_{n-1} D$ has trivial fibre
and both domain and codomain are 
$n$--polynomial. 
If we can show that $\tau_{n-1} D$
is connected at infinity, 
then we can conclude that $D$ is ${n-1}$--polynomial
by lemma \ref{lem:npolyfibre}.
This result boils down to showing that 
$\tau_{n-1}$ commutes with sequential homotopy colimits 
then noting that $D$ is connected (in fact contractible)
at infinity as $D$ is $n$--homogeneous. This commutation result
is simply a calculation and occurs
as \cite[Lemma 5.14]{weiss95}.

The functor $D$ is 
$n$--homogeneous as the sequential homotopy colimit
used to define $T_{n-1}$ will commute 
with the homotopy pullback used to define $D$. 
Hence $T_{n-1} D$ is objectwise contractible.
Since $D$ is $(n-1)$--polynomial, 
$D$ is levelwise equivalent to $T_{n-1} D$.
Thus $D$ is objectwise contractible. 
Another application of lemma \ref{lem:npolyfibre}
shows that $f$ is a objectwise weak equivalence
as desired. 
\end{proof}

\begin{rmk}
The weak equivalences for the $n$--homogeneous model structure 
are determined by the functor $\ind_0^n$. 
So is reasonable to ask if a weak equivalence
of $n$--polynomial objects is determined by the functors
$\ind_0^k$ for $0 \leqslant k \leqslant n$. 
The best approximation to this idea is 
\cite[Theorem 5.15]{weiss95}, which describes 
$T_n$--equivalences in terms of something akin to 
parametrised spectra. It would be valuable and interesting
to see if the techniques of \cite{ms06}
can be applied to expand upon this result. 
\end{rmk}

It is fascinating to see how closely the sets $S_n$ and $K_n$ are related to 
orthogonal spectra. The set $S_n$ comes from the cofibre sequence
of proposition \ref{prop:sgammatojn}. 
Inverting $S_n$ has the effect of 
killing the objects of $K_{n+1}$. As we will see shortly, 
the set of generating cofibrations for the $n$--stable 
model structure on $O(n) \ecal_n$ is also a localisation,
which has also killed the elements of $K_{n+1}$. Except that 
in this case we will use the cofibre sequence 
of proposition \ref{prop:jntojn}.

The case $n=0$ is much simpler than the rest. 
In particular, the $0$--homogeneous
model structure is equal to the $0$--polynomial model structure. 
A fibrant object in the $0$--polynomial model structures
is a homotopically constant
$\jcal_0$--space: the maps $X(V) \to X(V \oplus W)$ are weak equivalences
of spaces for all $V$ and $W$. 
A $T_0$--equivalence is a map $f \in \ecal_0$ such that
$\hocolim_k f(\rr^k)$ is a weak homotopy equivalence of spaces. 
A thorough study of this kind of model category appears in 
\cite[Section 15]{lind}. The discussion at the start of
\cite[Section 7]{weiss95} and 
\cite[Theorem 1.1]{lind} tell us that the homotopy
theory of $n \poly \ecal_0$ is the homotopy theory of based spaces. 
Thus we understand the $n=0$ case very well. 

For $n >0$, we want to understand the $n$--homogeneous objects, so we 
try to find a more structured model category 
that captures this homotopy theory, while having 
weak equivalences
that are simpler to understand.

\section{The \texorpdfstring{$n$}{n}-stable model structure}\label{sec:nstab}

We produce a model structure on the category $O(n) \ecal_n$ so 
that it is Quillen equivalent to $\ecal_0$ equipped with the $n$--homogeneous model structure. 
For this section, we are only 
interested in the case $n>0$, 
as the $0$--homogeneous model structure on $\ecal_0$ is 
the same as the $0$--polynomial model structure. 

We begin by relating $O(n) \ecal_n$ to  
orthogonal spectra, the primary difference being that
an object $X$ in $O(n) \ecal_n$ has structure maps of form
$X(V) \to \Omega^{nW} X(V \oplus W)$. 
Thus the model structure that we produce, which we call the 
$n$--stable model structure, is a variation of the usual stable
model structure to account for the unusual structure maps.  

To compare $O(n) \ecal_n$ with orthogonal spectra,
we apply the method of \cite{mm02} and \cite{mmss01}
for describing diagram spectra as diagram spaces. 
We want to reverse this process and describe $\jcal_n$--spaces as diagram spectra.

\begin{definition}
For $V$ a vector space, the one-point compactification of $V$ will be denoted by $S^V$, hence $S^{\rr^n} =S^n$. 
\end{definition}

\begin{definition}
For each $n \geqslant 0$, consider the functor $nS$ from $\ical$ to based spaces, which on an object $V$ takes value $nS(V) = S^{nV} =S^{\rr^n \otimes V}$. A map $f \co V \to W$ 
acts on the one-point compactification as $id \otimes f \co \rr^n \otimes V \to \rr^n \otimes W$.
\end{definition}
We see that $nS$ has an $O(n)$--action by acting on the $\rr^n$ factor. Note that $0S(V) = S^0$ for any $V$.

\begin{lemma}
For each $n \geqslant 0$, $nS$ is a commutative monoid in the category of $\ical$--spaces.
\end{lemma}
\begin{proof}
The multiplication is
$$\int^{A,B}
\ical(A \oplus B, V)_+ \smashprod S^{nA} \smashprod S^{nB}
\cong
\int^{A,B}
\ical(A \oplus B, V)_+ \smashprod S^{n(A \oplus B)}
\to S^{nV}
$$
where the last map is induced by the $\ical$--action map of $nS$.
\end{proof}

We also see that this multiplication $nS \smashprod nS \to nS$
is $O(n)$--equivariant, with the diagonal action on the domain. 

\begin{proposition}\label{prop:nsmodisjnspaces}
For each $n \geqslant 0$, the category $\ecal_n$ is equivalent to the category of $nS$--modules in $\ical$--spaces. Similarly the category $O(n)\ecal_n$ is the category of
$nS$--modules in $O(n)$--equivariant $\ical$--spaces.
\end{proposition}
\begin{proof}
An $nS$--module in the category of $\ical$--spaces is a topological 
functor $X$ from $\ical$ to based spaces with an action map 
$X \smashprod nS \to X$. 

Applying \cite{mmss01} to this data gives a topological category $\jscr_n$ such that $\jscr_n$--spaces are precisely $nS$--modules in $\ical$--spaces. The category $\jscr_n$ has the same objects as $\ical$ and the morphism spaces are given by
\[
\jscr_n(U,V) = nS \leftmod (U^* \smashprod nS, V^* \smashprod nS)
\]
where $U^*$ is the $\ical$--space $\ical(U,-)_+$. This expression can be reduced to
\[
\int^{A \in \ical}
\ical(A \oplus U, V)_+ \smashprod S^{nA}
\]
from which one can see that $\jscr_n$ is isomorphic to $\jcal_n$.
A specific isomorphism is induced by the map below 
\[
\begin{array}{rcl}
\ical(A \oplus U, V)_+ \smashprod S^{nA} 
& \longrightarrow &
\mor_n(U,V) \\
(f, x) 
& \longmapsto &
(f_{|U}, (\rr^n \otimes f)x)
\end{array}
\]

For the equivariant case, an $nS$--module in the category of $O(n)$--equivariant $\ical$--spaces is an $O(n)$--topological functor $Y$ from $\ical$ (with trivial action) to $O(n)$--spaces with an action map $Y \smashprod nS \to Y$, that is levelwise a map of $O(n)$--spaces. The calculation that the associated diagram category is isomorphic to $\jcal_n$ is similar to the non-equivariant case.
\end{proof}

Now we know that we have a category of diagram spectra, 
we can apply the rest of \cite{mmss01} to obtain a stable model structure. 
We now restrict ourselves to the case $n>0$. 

It is important to note that we
are using the `coarse model structure' on $O(n)$--spaces.
Here a map is a weak equivalence or fibration
if the underlying space map is so. The generating (acyclic)
cofibrations have form $O(n)_+ \smashprod i$
where $i$ is a generating (acyclic) cofibration for the model
category of spaces.

There is a levelwise model structure on the categories $O(n) \ecal_n$ for all $n \geqslant 0$, where the weak equivalences are the levelwise weak equivalences of underlying non-equivariant spaces.

\begin{definition}
In the category $O(n) \ecal_n$, a map $f \co E \to F$
is said to be a \textbf{levelwise fibration}
or a \textbf{levelwise weak equivalence}
if each $f(V) \co E(V) \to F(V)$ is a fibration or weak homotopy
equivalence of underlying spaces for each $V \in \ical$.
A \textbf{cofibration} is a map that has the left lifting property
with respect to all maps which are levelwise fibrations and
levelwise weak equivalences.
\end{definition}
Let $I_{\Top}$ and $J_{\Top}$ be the generating sets for the weak homotopy equivalence model structure on spaces. The following lemma, which is an application of \cite[Theorem 6.5]{mmss01}, gives the generating sets for the levelwise model structure on $O(n) \ecal_n$.

\begin{lemma}
The collections of cofibrations,
levelwise fibrations and levelwise weak equivalences
form a cellular, proper, topological model structure on the
category $O(n) \ecal_n$. We denote this model category by $O(n) \ecal_n^l$.
The generating sets are given below.
$$I = \{ \mor_n(V,-) \smashprod O(n)_+ \smashprod i | V \in \ical, \ i \in I_{\Top} \}$$
$$J = \{ \mor_n(V,-) \smashprod O(n)_+ \smashprod  j | V \in \ical, \ i \in J_{\Top} \}$$
\end{lemma}

We continue to follow the pattern for diagram spectra and construct the kinds of homotopy groups we need to consider in order to define an $n$--stable model structure. 

It is useful to note that a map $f \in O(n) \ecal_n$ is a levelwise fibration or weak equivalence if and only if $i^* f \in \ecal_n$ is. Similarly, our weak equivalences for the $n$--stable model structure will be defined in terms of underlying non-equivariant spaces. 

\begin{definition}
The \textbf{$n$--homotopy groups} of an object $X$ of $O(n) \ecal_n$are defined as
\[
n\pi_{k}(X) = {\colim_q} \pi_{nq+k}(X(\rr^q))
\]
A map is said to be an \textbf{$n\pi_*$--equivalence} if it induces isomorphisms
on $n$--homotopy groups for all integers $k$.
\end{definition}

\begin{lemma}
A levelwise weak equivalence of $O(n) \ecal_n$ is an $n\pi_*$--equivalence.
\end{lemma}

Now that we have the set of weak equivalences we are interested in, we should identify the fibrant objects of the $n$--stable model structure. The defining property of these fibrant objects should be that a $n \pi_*$--isomorphism between two fibrant objects is a levelwise weak equivalence. 

\begin{definition}
An object of $O(n) \ecal_n$ is an \textbf{$n\Omega$--spectrum} if the adjoints of its structure maps
$X(V) \to \Omega^{nW}X(V \oplus W)$ are weak homotopy equivalences.
\end{definition}

\begin{lemma}
An $n\pi_*$--equivalence between $n\Omega$--spectra is a levelwise weak  equivalence.
\end{lemma}
\begin{proof}
We want to show that
$\pi_k(f) \co \pi_k (X(V)) \to \pi_k (Y(V))$
is an isomorphism for all $V$ and all $k \geqslant 0$.
Choose an isomorphism $\rr^a \to V$, then
\[
\pi_q (X(V)) 
\cong 
\pi_q (X(\rr^a)) 
\cong 
\colim_{b} \pi_q (\Omega^{nb} X(\rr^{a+b}))
\]
Which is isomorphic to 
$n \pi_{q-na} (X)$.
Since $f$ is an $n \pi_*$--equivalence, it induces
an isomorphism
\[
\pi_q X(V) \cong n\pi_{q - na} (X) \to 
n\pi_{q - na} Y  \cong \pi_q Y(V)
\]
Hence $f(V) \co X(V) \to Y(V)$ is 
a weak equivalence of spaces. 
\end{proof}

We will use a left Bousfield localisation to create the 
$n$--stable model structure. Thus we need to identify a class of maps
which we will invert to create the $n$--stable model structure from the projective model structure. We will then need to show that this class is generated by a set and that this class coincides with the $n \pi_*$--isomorphisms. The general theory of localisations tells us that the class of weak equivalences is determined by the fibrant objects, in the sense of the following definition. 

\begin{definition}
We say that a map $f \co X \to Y$ is an \textbf{$n$--stable equivalence}
if the induced map on levelwise homotopy categories
$f^* \co [Y,E]_l \to [X,E]_l$ is an isomorphism for all
$n \Omega$--spectra $E$.
\end{definition}

Recall from proposition \ref{prop:jntojn} that, for inner-product spaces $V$ and $W$, there is a map
$S^{nW} \to \mor_n(V,V \oplus W)$ induced by sending 
$x \in nW$ to $(i,x) \in \gamma_n(V,V \oplus W)$,
where $i \co V \to V\oplus W$ is the standard inclusion.
By the Yoneda lemma, we can turn this into a map 
\[
\lambda_{V,W}^n \co 
\mor_n(V \oplus W,-) \smashprod S^{nW} 
\longrightarrow
\mor_n(V,-)
\]

\begin{lemma}\label{lem:lambdaconnect}
The maps $\lambda_{V,W}^n$ are $n$--stable equivalences and
$n\pi_*$--isomorphisms.
\end{lemma}
\begin{proof}
These maps have been chosen so that
\[
(\lambda_{V,W}^n)^* \co
\nat_{O(n)\ecal_n} (\mor_n(V,-),X) 
\longrightarrow
\nat_{O(n)\ecal_n} (\mor_n(V \oplus W,-)\smashprod  S^{nW},X)
\]
is precisely the adjoint of the structure map of $X$.
Thus they are automatically $n$--stable equivalences.
To see that they are $n\pi_*$--isomorphisms, we follow the usual 
calculation and see that the result holds as $\lambda_{V,W}^n(U)$
gets more and more connected as the dimension of $U$ increases.

If we fix a particular linear isometry $V \oplus W \to U$ (the 
choice is unimportant), we can use our description of $\mor_n(V,U)$ 
in proposition \ref{prop:nsmodisjnspaces} as a coend to write
\[
\lambda_{V,W}^n(U) \co 
O(U)_+ \smashprod_{O(U-V-W)} S^{n(U-V)}
\longrightarrow
O(U)_+ \smashprod_{O(U-V)} S^{n(U-V)}
\]

By the definition of $n\pi_*$, it is easy to see that 
this map is a $n \pi_*$--isomorphism if and only if its
suspension by $nV$ is. So we only need study 
$\Sigma^{nV} \lambda_{V,W}^n(U)$. The advantage of doing so is
that $O(U)$ will act
on the sphere term and we can rewrite it as
\[
\Sigma^{nV} \lambda_{V,W}^n(U) \co 
O(U)/{O(U-V-W)}_+ \smashprod S^{nU}
\longrightarrow
O(U)/{O(U-V)}_+ \smashprod S^{nU}
\]
This new map is $(n+1)\dim(U)-\dim(V)-\dim(W)$--connected.
So when we look at the $n\pi_{k}$--homotopy groups,
we are looking at maps from
$S^{nU +k}$ to a space that is
$(n+1)\dim(U)-\dim(V)-\dim(W)$--connected.
Since the dimension of $U$ increases in the colimit,
it is clear that we have a isomorphism of the colimits.
\end{proof}

The collection of $\lambda_{V,W}^n$ is the set of maps we wish to 
invert. We now need to turn these maps into cofibrations. Then
we can them to the generating acyclic cofibrations for the 
projective model structure to make a generating set for the 
$n$--stable model structure. 
 
Recall that the pushout product $f \square g$
of two maps $f \co A \to B$ and $g \co X \to Y$,
is defined to be the map
\[
f \square g \co 
A \smashprod Y \coprod_{A \smashprod X} B \smashprod X 
\longrightarrow B \smashprod Y
\]

\begin{definition}
Let $M\lambda_{V,W}^n$ be the mapping cylinder of $\lambda_{V,W}^n$ 
(which is homotopy equivalent to the codomain).
Let $k_{V,W}^n \co \mor_n(V \oplus W,-) \smashprod S^{nW} \to M\lambda_{V,W}^n$ be the inclusion into the top of the cylinder.
Now define
$$
J' = J \cup \{i \square k_{V,W}^n | i \in I_{\Top} \quad V,W \in \ical \}
$$
\end{definition}

\begin{proposition}
There is a cofibrantly generated, proper, topological model structure
on the category $O(n) \ecal_n$, called the \textbf{$n$--stable model structure}.
The cofibrations are as for the levelwise model structure, 
the weak equivalences are the
$n\pi_*$--isomorphisms and the fibrant objects are the $n\Omega$--spectra. This model category is written $O(n) \ecal_n^{\pi}$.
\end{proposition}

The proof of this result is all but identical to \cite{mmss01} or \cite{mm02}.
As an illustration, we will identify the fibrations of this model
structure. But first, we want to justify the use of the term stable with the following lemma, which is proved in the same manner as for other categories of diagram spectra.

\begin{lemma}
A map $f$ in $O(n) \ecal_n$ is an $n\pi_*$--equivalence
if and only if $\Sigma f$
is an $n\pi_*$--equivalence.
\end{lemma}

\begin{lemma}
A map $f \co E \to B$ has the right-lifting-property with respect to $J'$
if and only if $f$ is an levelwise fibration and the diagram below is always a
homotopy pullback.
$$\xymatrix{
E(V) \ar[r] \ar[d] & \Omega^{nW} E(V \oplus W) \ar[d] \\
B(V) \ar[r]        & \Omega^{nW} B(V \oplus W)
}$$
Thus the fibrant object are precisely the $n \Omega$--spectra.
\end{lemma}
\begin{proof}
Assume that $f$ has the right-lifting-property with respect to $J'$,
then it is certainly an levelwise fibration.
So we must check that $f$ has the right-lifting-property with respect to
$i \square k_{V,W}$. This is equivalent to checking that
$O(n) \ecal_n (k_{V,W}^*, p_*)$ is an acyclic fibration of spaces.
We know that $k_{V,W}$ is a cofibration of 
$\jcal_n$--spaces and $p$ is a fibration,
so it suffices to show that
$O(n) \ecal_n (k_{V,W}^*, p_*)$ is a weak equivalence.
By the way we have constructed $k_{V,W}$
all we need to show is that
$O(n) \ecal_n ((\lambda_{V,W}^n)^*, p_*)$
is a weak equivalence. Writing out what this means is
precisely the statement that the diagram of the lemma
is a homotopy pullback.
Carefully reading this argument shows that
the converse is also true.
\end{proof}

\begin{corollary}
A map $f$ in $O(n) \ecal_n$ is an $n\pi_*$--equivalence
if and only if it is an $n$--stable equivalence.
\end{corollary}

\section{The equivalence between 
\texorpdfstring{$O(n)\ecal_n$}{O(n)E\_n}
and 
\texorpdfstring{$O(n)\ical \scal$}{O(n)IS}
}\label{sec:equiv}

We prove that the category $O(n) \ecal_n$ is Quillen 
equivalent to the category of $O(n)$--objects in the category of 
orthogonal spectra. The basic idea for this section (and the model structure of the previous section) comes from 
\cite[Section 3]{weiss95}. In that section Weiss constructs 
a spectrum $\Theta E$  
from the data of a `symmetric' object $E$ in $\ecal_n$ (such an 
object is precisely an object of $O(n) \ecal_n$). 
His notion of equivalence of spectra then corresponds
to our notion of $n\pi_*$--isomorphism. 

The essential concept is that if one has a spectrum 
$X$ with an action of $O(n)$, one can make an object $O(n)\ecal_n$, 
which at $V$ takes value $X(\rr^n \otimes V)$. 

\begin{definition}
The category of $O(n)$--equivariant objects in orthogonal spectra, 
written as $O(n) \ical \scal$, is the category of $O(n)$--objects 
in $\ecal_1$ and $O(n)$--equivariant maps. This category has a 
cofibrantly generated, proper, stable model structure where 
a map $f$ is a weak equivalence or fibration if and only if 
it is so when considered as a non-equivariant map 
in the stable model structure on orthogonal spectra. 
\end{definition}

Thus an object $X \in O(n) \ical \scal$ is a continuous
functor $X$ from $\jcal_1$ to based topological spaces such that
there is a group homomorphism from $O(n)$ into 
$\ecal_1 (X,X)$. It follows that each space $X(V)$ has a group 
action and the structure maps $S^W \smashprod X(V) \to X(V \oplus W)$
are $O(n)$--equivariant ($S^W$ has the trivial action). 
We can also describe $X$ as a functor 
$X \co \varepsilon^* \jcal_1 \to O(n) \Top$ of categories 
enriched over $O(n) \Top$.
Hence we have a continuous map 
\[
X(V,W) \co \jcal_1(V,W) \to \Top (X(V), X(W))^{O(n)}
\]

We want to find a map of enriched categories 
$\alpha_n \co \jcal_n \to \jcal_1$ so that any object $X$ of 
$\ecal_1$ can be made into an object $X \circ \alpha_n$ of 
$\ecal_n$.  We will then deal with equivariance and show that the 
$O(n)$--action on $X$ turns $X \circ \alpha_n$ into an object of 
$O(n) \ecal_n$. 

\begin{definition}
There is a map of topological categories 
$\alpha_n \co \jcal_n \to \jcal_1$
which sends the object $U$ to $\rr^n \otimes U = nU$
and on morphism spaces acts as 
\[
\begin{array}{rcl}
\alpha_n(U,V) \co \jcal_n(U,V) 
& \longrightarrow 
& \jcal_1(nU, nV) \\
(f,x) 
& \longmapsto 
&(\rr^n \otimes f, x)
\end{array}
\]
\end{definition}

Now consider some $X \in O(n) \ical \scal$, 
we have $X \circ \alpha_n$, a continuous functor from $\jcal_n$ to 
based spaces. The space $X(nU)$ has two different $O(n)$--actions 
on it, the first comes from the fact that for any $V$, $X(V)$
is an $O(n)$--space. For $\sigma \in O(n)$
we denote this self-map of $X(V)$ as 
$\sigma_{X(V)} \co X(V) \to X(V)$.
The second action comes from 
thinking of an element $\sigma$ of $O(n)$ as a map
\[
\sigma \otimes U \co \rr^n \otimes U \to \rr^n \otimes U
\]
Thus we have an element $(\sigma \otimes U,0) \in \jcal_1(nU,nU)$
applying $X$ to this gives a self-map of $X(nU)$, which we call
$X(\sigma \otimes U)$. 
By the definition of $X$, these two actions commute.

Now we are ready to make $\alpha_n^* X \in O(n) \ecal_n$. 
Ignoring equivariance, it is just the composite functor
$X \circ \alpha_n$. 
Hence at $V$, $(\alpha_n^* X)(V) = X(nV)$. 
Now we must equip it with an $O(n)$--action, 
we let $\sigma \in O(n)$ act on 
$(\alpha_n^* X)(V)$ by 
$X(\sigma \otimes U) \circ \sigma_{X(U)}$.
Thus $\alpha_n^* X$ takes values in $O(n)$--spaces. 
We must now check that the map
\[
\alpha_n^* X \co 
\jcal_n(U,V) \longrightarrow
\Top(\alpha_n^* X(U), \alpha_n^* X(V) ) 
\]
is an $O(n)$--equivariant map. 
This will imply that the structure map
below is $O(n)$--equivariant, where the domain has the diagonal 
action. Hence $\alpha_n^* X$ will be an object of $O(n) \ecal_n$. 
\[
S^{nV} \smashprod (\alpha_n^* X)(U)
\longrightarrow
(\alpha_n^* X)(U \oplus V)
\]

It is routine to check that the following diagram commutes.
\[
\xymatrix@C+0.3cm{
\mor_n(U,V) 
\ar[r]^{\alpha_n} \ar[d]^\sigma &
\mor_1(nU,nV) 
\ar[r]^(0.45){X} 
\ar[d]^{(\sigma^{-1} \otimes U)^*
\circ (\sigma \otimes V)_*} &
\Top(X(nU),X(nV)) 
\ar[d]^{(X(\sigma^{-1} \otimes U))^* 
\circ (X(\sigma \otimes V))_*} \\
\mor_n(U,V) 
\ar[r]^{\alpha_n}&
\mor_1(nU,nV) 
\ar[r]^(0.45){X} &
\Top(X(nU),X(nV)) 
}
\]
Applying $X \circ \alpha_n$ to any pair $(g,y) \in \mor_1(W,W')$
gives an $O(n)$--equivariant map $X(nW) \to X(nW')$, it follows that
the two expressions below are equal.
\[
\begin{array}{c}
X(\sigma \otimes W') \circ 
X(\rr^n \otimes g,y) \circ 
X(\sigma^{-1} \otimes W) \\
\phantom{n} \sigma_{X(W)} \circ 
X(\sigma \otimes W') \circ 
X(\rr^n \otimes g,y) \circ 
X(\sigma^{-1} \otimes W) \circ 
\sigma^{-1}_{X(W')}
\end{array}
\]

We now describe a left adjoint to $\alpha_n^*$. 
We can think of $\jcal_1(nU,V)$ as a topological
space with a left action of $\jcal_1$ and a right action
of $\jcal_n$. We can use this to put a $\jcal_1$ action onto
an object of $O(n) \ecal_n$. So let $Y \in O(n) \ecal_n$, then 
(using enriched coends) we define 
\[
(Y \smashprod_{\jcal_n} \jcal_1)(V) = \int^{U \in \jcal_n}
Y(U) \smashprod \jcal_1(nU, V)
\]
To see that this is the left adjoint of $\alpha_n^*$ 
is a formal exercise in manipulating coends.

\begin{proposition}\label{prop:eqspectra}
The adjoint pair 
\[
(-) \smashprod_{\jcal_n} \jcal_1 :
O(n) \ecal_n 
\overrightarrow{\longleftarrow}
O(n) \ical \scal
: \alpha_n^*
\]
is a
Quillen equivalence between $O(n) \ecal_n$ equipped with the 
$n$--stable model structure and $O(n) \ical \scal$.
\end{proposition}
\begin{proof}
It is routine to check that $\alpha_n^*$ preserves levelwise 
fibrations, levelwise trivial fibrations and takes 
fibrations in $O(n) \ecal_n^{\pi}$ to fibrations in 
$O(n) \ical \scal$. 

A straightforward argument using cofinality shows that 
a map $f \in O(n) \ical \scal$ is 
a $\pi_*$--isomorphism if and only if 
$\alpha_n^* f$ is an $n \pi_*$--isomorphism.  
All that remains 
is to show that the derived unit is an 
$n \pi_*$--isomorphism.  
Since our categories are stable, it suffices to do so for the 
generator of $O(n) \ecal_n$, which is $O(n)_+ \smashprod nS$.  
Writing $nS$ as $\mor_n(0,-)$, it is easy to see that 
\[
O(n)_+ \smashprod \mor_n(0,-) \smashprod_{\jcal_n} \jcal_1  
= O(n)_+ \smashprod \mor_1(0,-)
\]
and that $\alpha_n^*$ of this is 
$O(n)_+ \smashprod \mor_n(0,-)$. 
\end{proof}

We note that $\alpha_n^*$ also has a right adjoint, 
given in terms of an enriched end. However we have not been able 
to show that this right adjoint is part of a Quillen pair, hence we 
cannot pass directly from $\ecal_0$ to $O(n) \ical \scal$. 
Instead, we have a zig-zag of Quillen pairs, which is 
commonplace when working with model categories and no real disadvantage.

Now that we fully understand the $n$--stable model structure on 
$O(n) \ecal_n$, we should compare
it with the $n$--polynomial and $n$--homogeneous model structures on
$\ecal_0$.

\section{Inflation--induction as a Quillen functor}\label{sec:infind}

In this section we show that inflation--induction and 
restriction--orbits form a
Quillen pair between the $n$--stable model structure on $O(n) 
\ecal_n$ and $\ecal_0$ equipped with either the $n$--polynomial
or the $n$--homogeneous model structures.

\begin{lemma}
For $n \geqslant 0$, there is a Quillen adjunction
\[
\res_0^n/O(n) : O(n) \ecal_n^l
\overrightarrow{\longleftarrow}
\ecal_0
: \ind_0^n \varepsilon^*
\]
\end{lemma}
\begin{proof}
A generating cofibration takes form
$\mor_n(V,-) \smashprod O(n)_+ \smashprod i$, where $i$ is a 
generating
cofibration for the model category of based spaces. Applying the
left adjoint to this gives
$\mor_n(V,-) \smashprod i$. 
By lemma \ref{lem:coffunct}, 
$\mor_n (V,-)$ is a cofibrant object of $\ecal_0$. It follows 
that the left adjoint preserves cofibrations. 
The case of acyclic cofibrations is identical.
\end{proof}

\begin{lemma}
For $n \geqslant 0$, there is a Quillen adjunction
$$
\res_0^n/O(n) : O(n) \ecal_n^{\pi}
\overrightarrow{\longleftarrow}
n \poly \ecal_0
: \ind_0^n \varepsilon^*
$$
\end{lemma}
\begin{proof}
The identity functor is a left Quillen functor from 
$\ecal_0$ to $n \poly \ecal_0$. So 
$\res_0^n/O(n)$ is a left Quillen functor from 
$O(n) \ecal_n$ to the $n$--polynomial model structure on $\ecal_0$. 
To check that this Quillen functor passes 
to the $n$--stable model structure we apply 
\cite[Theorem 3.1.6 and Proposition 3.1.18]{hir03}.
We must show that $\ind_0^n \varepsilon^*$
takes $n$--polynomial objects of $\ecal_0$ to 
$n$--stable objects of $O(n) \ecal_n$. 
We have done so in proposition \ref{prop:npolynstable}.
\end{proof}

Composing this adjunction with the change of 
model structures adjunction between 
$n \poly \ecal_0$ and $m \poly \ecal_0$, for $n > m$, 
gives a Quillen pair between
$O(n) \ecal_n^{\pi}$ and $m \poly \ecal_0$. If $X$ 
is an $m$--polynomial functor
then, by proposition \ref{prop:mpolynpoly}, we know that it is 
$(n-1)$--polynomial. Hence $\ind_0^n \varepsilon^* X$ is levelwise 
contractible by corollary \ref{cor:npolydiff}. Therefore, 
on homotopy categories, the derived functor of 
$\ind_0^n \varepsilon^*$ sends every object of 
$\ho(m \poly \ecal_0)$ to the terminal object of 
$\ho( O(n) \ecal_n^{\pi})$.

\begin{lemma}\label{lem:derivedfun}
The left derived functor of $\res_0^n /O(n)$
is levelwise weakly equivalent to 
$EO(n)_+ \smashprod_{O(n)} \res_0^n (-)$.
\end{lemma}
\begin{proof}
Let $X$ be an object of $O(n) \ecal_n^{\pi}$, then 
$\cofrep X \to X$ is a levelwise acyclic fibration
and $\cofrep X$ is levelwise free. Hence the two maps
below are objectwise weak equivalences of $\ecal_0$. 
\[
\begin{array}{rcl}
EO(n)_+ \smashprod_{O(n)} \res_0^n (\cofrep X) 
& \longrightarrow &
\res_0^n (\cofrep X)/ O(n)\\
EO(n)_+ \smashprod_{O(n)} \res_0^n (\cofrep X) 
& \longrightarrow &
EO(n)_+ \smashprod_{O(n)} \res_0^n X
\end{array}
\]
\end{proof}

\begin{lemma}
For $n \geqslant 0$ there is a Quillen adjunction
$$
\res_0^n/O(n) : O(n) \ecal_n^{\pi}
\overrightarrow{\longleftarrow}
n \homog \ecal_0
: \ind_0^n \varepsilon^*
$$
\end{lemma}
\begin{proof}
We know that a map $f$ 
is a weak equivalence 
between fibrant objects
in the $n$--homogeneous model structure if and only if 
$\ind_0^n \varepsilon^* f$
is a levelwise weak equivalence. 
Hence this adjunction is a 
Quillen pair by 
\cite[Proposition 3.1.18]{hir03}.
\end{proof}

We draw these Quillen pairs to show how they are related, left 
adjoints will be either on the top, or on the left of a pair.  
\[
\xymatrix@C+1cm{
O(n)\ecal_n^{l}
\ar@<+1ex>[r]^(0.55){\res_0^n/O(n)}
\ar@<-1ex>[d]_{1}
&
\ecal_0
\ar@<+1ex>[l]^(0.4){\ind_0^n \varepsilon^*}
\ar@<+1ex>[r]^(0.4){1}
&
n \poly \ecal_0
\ar@<+1ex>[l]^(0.6){1}
\ar@<+1ex>[d]^{1}   
\\
O(n)\ecal_n^{\pi}
\ar@<+1ex>[rr]^{\res_0^n/O(n)}
\ar@<-1ex>[u]_{1}
& 
&
n \homog \ecal_0
\ar@<+1ex>[ll]^{\ind_0^n \varepsilon^*}
\ar@<+1ex>[u]^{1}   \\
}
\]

\section{The classification of \texorpdfstring{$n$}{n}-homogeneous functors}\label{sec:tower}

Now we prove that the homotopy category of $O(n)\ecal_n^{\pi}$ is the
homotopy category of $n$--homogeneous functors. 
Throughout this section we will keep the diagram of Quillen functors below in mind.
\[
\xymatrix@C+1cm{
n \homog \ecal_0
\ar@<-1 ex>[r]_(0.6){\ind_0^n \varepsilon^*}
&
O(n)\ecal_n^{\pi}
\ar@<-1ex>[l]_(0.4){\res_0^n/O(n)}
\ar@<+1ex>[r]^{(-)\smashprod_{\jcal_n}\jcal_1}
&
O(n)\ical \scal
\ar@<+1ex>[l]^{\alpha_n^*}
}
\]

\begin{theorem}
For $n \geqslant 0$, the Quillen adjunction
$$
\res_0^n/O(n) : O(n) \ecal_n^{\pi}
\overrightarrow{\longleftarrow}
n \homog \ecal_0
: \ind_0^n \varepsilon^*
$$
is a Quillen equivalence
\end{theorem}
\begin{proof}
The derived functor of the right Quillen $ \ind_0^n \varepsilon^*$ reflects equivalence in the $n$--homogeneous structure.
Thus the only thing to show is that the derived unit is an equivalence in the stable structure. 

Given some cofibrant $X \in O(n) \ecal_n$, by proposition \ref{prop:eqspectra}, we have an $n \pi_*$--isomorphisms
\[
X \longrightarrow X_f \longrightarrow 
\alpha_n^* \fibrep (X \smashprod_{\jcal_n} \jcal_1  )
\]
where $\fibrep$ denotes fibrant replacement in the stable model 
category on $O(n) \ical \scal$. Let $\Psi$ denote 
$\fibrep X \smashprod_{\jcal_n} \jcal_1$, this is an $\Omega$--spectrum with
an action of $O(n)$. 
Let $\cofrep$ denote cofibrant replacement in $O(n) \ecal_n^\pi$, 
then one can construct a commutative square
\[
\xymatrix{
X \ar[r]^1 \ar[d]^2 &
\cofrep \alpha_n^* \Psi \ar[d]^3 \\
\ind_0^n \varepsilon^* T_n \res_0^n X/O(n)
\ar[r]^4 &
\ind_0^n \varepsilon^* T_n \res_0^n (\cofrep \alpha_n^* \Psi ) /O(n) \\
}
\]
We want to know that map 2 is an $n$--stable equivalence. 
Map 1 is a stable equivalence, as is 4, since 
derived functors preserve all weak equivalences.
Thus we must show that map 3 is a stable equivalence. 
Lemma \ref{lem:derivedfun} tells us that the codomain of 
map 3 is levelwise weakly equivalent to 
\[
\ind_0^n \varepsilon^* T_n (EO(n)_+ \smashprod_{O(n)} \res_0^n 
(\alpha_n^* \Psi )
\]
The object $T_n (EO(n)_+ \smashprod_{O(n)} \res_0^n 
(\alpha_n^* \Psi ) \in \ecal_0$ is identified in 
\cite[Example 6.4]{weiss95}. 
It is shown to be equivalent to 
the object $F \Psi$ of $\ecal_0$ defined by
\[
V \longmapsto \hocolim_k \Omega^{nk} 
[EO(n)_+ \smashprod_{O(n)} (\Psi(\rr^k) \smashprod S^{nV})]
\]
where $\theta \smashprod S^{nV}$ has the diagonal action. 
This calculation is performed via a connectivity argument based on 
the interplay between homotopy--orbits and the functor 
$\Omega^\infty \Sigma^\infty$ on equivariant spaces. 

The derived $n$--stable equivalence then follows from the 
calculation of 
$\ind_0^n \varepsilon^* F \Psi$, as 
given in \cite[Example 5.7]{weiss95}. 
This example proves that 
$\ind_0^m \varepsilon^* F \Psi$ is given by
\[
V \longmapsto \hocolim_k \Omega^{nk} 
[EO(n-m)_+ \smashprod_{O(n-m)} (\Psi(\rr^k) \smashprod S^{nV})]
\]
We are interested in the case $m=n$, where we see that 
$\hocolim_k \Omega^{nk} (\Psi(\rr^k) \smashprod S^{nV})$
is weakly equivalent to $(\alpha_n^* \Psi)(V)$. 
\end{proof}

We wish to remark that despite extensive efforts, the authors
have been unable to improve upon the two examples
quoted above.
It is interesting to note how closely the derivatives 
of $F \Psi$ are related to the 
generalised restriction--orbit functors
$\res_m^n \Psi/O(n-m)$.

Composing the right derived functor of $\ind_0^n \varepsilon^*$ with the left derived functor of $(-)\smashprod_{\jcal_n}\jcal_1$ recovers the classification in \cite[Theorem 7.3]{weiss95}.

\begin{corollary}
There is an equivalence
of homotopy categories:
$$\xymatrix@C+1.6cm{
\ho (O(n)\ical \scal)
\ar@<+1ex>[r]
&
\ho (n\homog \ecal_0)
\ar@<+1ex>[l]
}
$$
\end{corollary}

Now we can show how an $n$--polynomial object $X$ is made from an
$(n-1)$--polynomial object and an $n$--homogeneous object. 
The following is \cite[Theorem 9.1]{weiss95}.
For $X \in \ecal_0$, inflation--induction and the left adjoint of $\alpha_n^*$
determine an object of $O(n) \ical \scal$, 
denote this object by $\Psi^n_X$.
 
\begin{theorem}\label{thm:fibre}
For any $X \in \ecal_0$, $n>0$ and $V \in \jcal_0$, there is a homotopy fibration sequence
\[
\Omega^\infty [EO(n)_+ \smashprod_{O(n)} (\Psi^n_X \smashprod S^{nV})]
\longrightarrow (T_n X)(V)
\longrightarrow (T_{n-1} X)(V)
\]
\end{theorem}

We now give the picture of the tower for $X \in \ecal_0$. 
$$
\xymatrix@C=1cm@R-0.3cm{
&& \ar[d]\\
&& T_3 X \ar[d]
&\Omega^\infty [EO(3)_+ \smashprod_{O(3)} (\Psi^3_X \smashprod S^{3V})] \ar[l]
\\
&& T_2 X \ar[d]
& \Omega^\infty [EO(2)_+ \smashprod_{O(2)} (\Psi^2_X \smashprod S^{2V})] \ar[l]
\\
&& T_1 X \ar[d]
& \Omega^\infty [EO(1)_+ \smashprod_{O(1)} (\Psi^1_X \smashprod S^{1V})] \ar[l]
\\
X \ar[rr] \ar[urr] \ar[uurr] \ar[uuurr]
&& T_0 X }
$$

We have completed our task: the constructions of \cite{weiss95} 
are now realised as Quillen functors on model categories,
we have classified homogeneous functors via a Quillen equivalence, the relation between homogenous functors and $O(n)$--spectra
has been made precise
and orthogonal calculus is ready to be generalised to equivariant 
or stable settings.

\section{Stable Orthogonal Calculus}\label{sec:stablecase}
As an application, we outline a stable variant of orthogonal calculus, which
replaces based topological spaces with
orthogonal spectra.  
As $\Sigma^{\infty}$ is a symmetric monoidal functor,
\cite[Lemma II.4.8]{mm02},
 we can think of the categories $\jcal_n$ as enriched over spectra; 
in detail, we make a new category $\kcal_n$, with the same objects 
as $\jcal_n$, whose hom-objects are given by
\[
\kcal_n (U,V) = \Sigma^\infty \mor_n (U,V)
\]

We define the category $\fcal_n$ to be the category of 
$O(n) \ical \scal$--enriched functors from 
$\kcal_n$ to $O(n) \ical \scal$. Notice this is the same category as
taking  $O(n) \Top$--enriched functors
from $\jcal_n$ to $O(n) \ical \scal$.

The required adjunctions which define the notion of differentiation
are simply enriched Kan extension. Thus changing the codomain to spectra causes no additional complications.  In particular, inflation--induction   
$ \ind_m^n:O(m)\fcal_m \rightarrow O(n)\fcal_n$ is defined as
$$(\ind_m^n \CI F)(V) = \nat_{O(m)\fcal_m} (\kcal_n(V,-), \CI_m^n F)$$
where $\CI_m^n$ is given by the adjunction 
$$(-)/O(n-m) : O(n)\ical \scal
\overrightarrow{\longleftarrow}
O(m) \ical \scal : \CI_m^n $$

For fibrant $E \in \fcal_0$, define $\tau_n E \in \fcal_0$ by
$$(\tau_n E)(V) = \nat_{\fcal_0}( \Sigma^{\infty} S \gamma_{n+1}(V,-)_+, E)$$

At this point  we note that $\Sigma^{\infty}:\Top \rightarrow  \ical \scal $ is left Quillen.   In particular, this implies the alternative description of $\tau_n$ still holds for fibrant $E\in \fcal_0$:
\[
\tau_n E(V) = \underset{0 \neq U \subset \rr^{n+1}}{\holim} 
E(U \oplus V)
\]

\begin{definition}\label{def:specpoly}
A functor $E$ from $\kcal_0$ to spectra
is said to be \textbf{polynomial of degree less than or equal to $n$} if and only if
each spectrum $E(V)$ is an $\Omega$-spectrum and
$$
(\rho_n)_E \co E \to \tau_n E
$$
is a levelwise stable equivalence. This second condition is equivalent 
to asking  that for each $V$, 
$$
E(V)\rightarrow \underset{0 \neq U \subset \rr^{n+1}}{\holim} 
E(U \oplus V)
$$
is a weak equivalence of spectra.
\end{definition}

In addition, since $\Sigma^{\infty}:\Top \rightarrow  \ical \scal $ is left Quillen,  we retain the homotopy cofibre sequences below.
\[
\begin{array}{ccccc}
\kcal_n(\rr \oplus V, W) \smashprod S^n
& \longrightarrow &
\kcal_n (V,W)
& \longrightarrow &
\kcal_{n+1} (V,W) \\ 
\Sigma^{\infty} S\gamma_{n+1}(V,W)_+  
& \longrightarrow &
\kcal_0 (V,W) 
& \longrightarrow &
\kcal_{n+1}(V,W)
\end{array}
\]

In particular, note that $E \in \fcal_0$ is $n$--polynomial if and only if 
$\ind_0^{n+1} E (V)$ is $\pi_*$--isomorphic to a point, 
for each $V \in \kcal_0$. 

For $i \geqslant 0$, we build the fibrant replacement functor
$T_i$ from $\tau_i$ exactly as in definition \ref{def:tn}.

\begin{definition}\label{def:spechomog}
An object $E \in \ecal_0$ is \textbf{$n$--homogeneous} if it is
polynomial of degree at most $n$ and $T_{n-1} E$ is
weakly equivalent to the terminal object.  
\end{definition}

Since $O(n) \ical \scal$, $n \geqslant 0$,  is a spectra-enriched (or 
topologically enriched) cellular model category, we may apply our previous work 
and obtain all the necessary model structures.  Specifically, we obtain 
the $n$--polynomial and $n$--homogeneous model structures on $\fcal_0$,  
as well as the $n$--stable model structure on $O(n) \fcal_n$, in precisely the same manner.  

We thus have a tower of fibrations relating $T_n E$
to $T_{n-1} E$ for varying $n$. The utility of the tower comes from the 
identification of the homotopy fibres in terms of inflation--induction, 
so that is our next task.

\begin{theorem}
For $n \geqslant 1$, the adjoint functors
\[
\res_0^n/O(n) : O(n) \fcal_n^{\pi}
\overrightarrow{\longleftarrow}
n \homog \fcal_0
: \ind_0^n \varepsilon^*
\]
determine a Quillen equivalence.
\end{theorem}

In particular, by simply replacing spectra for spaces, we obtain a zig-zag of Quillen equivalences, where  $O(n)(\ical \times \ical) \scal$ 
is the category of $O(n)$--objects in $\fcal_1$.
\[
\xymatrix@C+1cm{
n \homog \fcal_0
\ar@<+1 ex>[r]^(0.55){\ind_0^n \varepsilon^*}
&
O(n)\fcal_n^{\pi}
\ar@<+1ex>[l]^(0.45){\res_0^n/O(n)}
\ar@<+1ex>[r]^(0.42){(-)\smashprod_{\kcal_n}\kcal_1}
&
O(n)(\ical \times \ical) \scal 
\ar@<+1ex>[l]^(0.6){\alpha_n^*}
}
\]
We can easily go further by noting that this category is the 
category of $O(n)$--objects in orthogonal bi-spectra. 
Following \cite[Corollary 2.30]{jar97}, 
the ``diagonal" functor
$$d:(\ical \times \ical) \scal \rightarrow \ical \scal$$
as a right adjoint, defines an  equivalence of the associated homotopy 
categories.  In addition, the diagonal preserves the 
tensor product of a spectrum with a space, 
so the equivalence lifts to an equivalence of homotopy categories 
\[
d_*:Ho(O(n)(\ical \times \ical) \scal) 
\overset{\simeq}{\longrightarrow} Ho(O(n)\ical \scal)
\]
It should be routine to extend this to a Quillen equivalence, 
provided one takes care over the categorical foundations. 

In summary, we have the following identification of homogeneous functors.
\begin{corollary}
There is an equivalence
of homotopy categories:
$$\xymatrix@C+1.6cm{
\ho (O(n)\ical \scal)
\ar@<+1ex>[r]
&
\ho (n\homog \fcal_0)
\ar@<+1ex>[l]
}
$$

For $F \in \fcal_0$, we denote the corresponding object in the homotopy category of $O(n)\ical\scal$ as $\Phi^n_F$.
\end{corollary}

Having identified the fibres via a Quillen equivalence and in terms of orthogonal spectra, we have obtained the desired stable variant of orthogonal calculus.

\begin{theorem}
A functor $F \in \fcal_0$ determines a tower of fibrations: 
$$
\xymatrix@C=1cm@R-0.3cm{
&& \ar[d]\\
&& T_3 F \ar[d]
&  [EO(3)_+ \smashprod_{O(3)} (\Phi^3_F \smashprod S^{3V})]
\ar[l]
\\
&& T_2 F \ar[d]
&  [EO(2)_+ \smashprod_{O(2)} (\Phi^2_F \smashprod S^{2V})]
\ar[l]
\\
&& T_1 F \ar[d]
&[EO(1)_+ \smashprod_{O(1)} (\Phi^1_F \smashprod S^{1V})]
\ar[l]
\\
F \ar[rr] \ar[urr] \ar[uurr] \ar[uuurr]
&& T_0 F }
$$
\end{theorem}

One can now study interactions between orthogonal calculus and stable homotopy theory, particularly interactions with various localizations of spectra, which the authors believe to be intractable without the current work. 
For example, let $X$ be an object of $\fcal_0$ such that $T_0 X$ is objectwise contractible. 
If we look at the derivatives of $X$ rationally, then we know by 
\cite{greshifree} that the $n^{th}$ derivative is given by 
a torsion module over the twisted group ring $\text{H}^*(\text{B}SO(n))[C_2]$.

\bibliography{orthorewrite}
\bibliographystyle{alpha}

\begin{tabular}{l}
David Barnes \\
School of Mathematics and Statistics \\
Hicks Building \\ 
Sheffield S3 7RH, UK \\
d.j.barnes@shef.ac.uk \\
\\
\\
Peter Oman \\
Department of Mathematics \\
Southeast Missouri State University \\
Cape Girardeau, MO 63701, USA \\
poman@semo.edu
\end{tabular}

\end{document}